\documentclass[a4paper,12pt,twoside]{amsart}

\usepackage{fullpage}
\usepackage{mymacros}
\usepackage{oip}

\title[Mirror symmetry]{Residue mirror symmetry for Grassmannians}
\author[B.~Kim]{Bumsig Kim}
\address{
Korea Institute for Advanced Study,
85 Hoegi-ro
Dondaemun-gu
Seoul 02455,
Republic of Korea
}
\email{bumsig@kias.re.kr}
\author[J.~Oh]{Jeongseok Oh}
\address{
Korea Institute for Advanced Study,
85 Hoegi-ro
Dondaemun-gu
Seoul 02455,
Republic of Korea
}
\email{batistuta@kaist.ac.kr}
\author[K.~Ueda]{Kazushi Ueda}
\address{
Graduate School of Mathematical Sciences,
The University of Tokyo,
3-8-1 Komaba
Meguro-ku
Tokyo
153-8914
Japan.}
\email{kazushi@ms.u-tokyo.ac.jp}
\author[Y.~Yoshida]{Yutaka Yoshida}
\address{
Kavli IPMU (WPI), UTIAS, University of Tokyo
Kashiwa, Chiba 277-8583, Japan
}
\email{yyyyosida@gmail.com}
\date{}
\begin{document}

\maketitle

\begin{abstract}
Motivated by recent works on localizations
in A-twisted gauged linear sigma models,
we discuss a generalization of toric residue mirror symmetry
to complete intersections in Grassmannians.
\end{abstract}

\setcounter{tocdepth}{1}
\tableofcontents

\section{Introduction}

\emph{A-twisted gauged linear sigma models}
are 2-dimensional topological field theories
introduced by Witten \cite{MR1232617}.
An A-twisted gauged linear sigma model
is specified
by a reductive algebraic group $G$
(or its compact real form)
called the \emph{gauge group},
an affine space $W$
with a linear action of $G \times \bGm$
called the \emph{matter},
and an element $\xi$
of the dual
$\frakz^*$
of the center of the Lie algebra
of $G$
called the \emph{Fayet--Iliopoulos parameter}.
The weights of the $\bGm$-actions are called \emph{R-charges}.
One can also introduce a \emph{superpotential} in the theory,
which is a $G$-invariant function on $W$ of R-charge 2.
The \emph{correlation functions},
which are quantities of primary interest,
do not depend on the potential.

An A-twisted gauged linear sigma model
with a suitable Fayet--Iliopoulos parameter
is expected to be equivalent
to the topological sigma model
whose target is the classical vacuum subspace of the symplectic reduction $W \GIT_\xi G$.
This comes from a stronger expectation
that the low-energy limit of a gauged linear sigma model
should give the non-linear sigma model
whose target is the symplectic reduction $W \GIT_\xi G$.

A prototypical example is the case
$G = \bGm$ and $W = \bA^6$,
with the action
\begin{align}
 G \times \bGm \ni (\alpha, \beta) \colon
  (z_1, \ldots, z_5, P) \mapsto (\alpha z_1, \ldots, \alpha z_5, \alpha^{-5} \beta^2 P)
\end{align}
and a potential $P f$,
which is the product 
of the variable $P$ and
a homogeneous polynomial $f$ in $z_1, \ldots, z_5$ of degree 5.
The symplectic reduction $W \GIT_\xi G$
for the positive $\xi$ gives the total space
of the bundle $\cO_{\bP^4}(-5)$.
The R-charge of the $P$-field indicates
that the target space should be considered
not as a manifold
but as a supermanifold,
where the parity of the fiber is odd.

One candidate for a mathematical theory
of A-twisted gauged linear sigma models
is symplectic vortex invariants
\cite{MR1777853,MR1959059,
MR1953239,MR2554933,
MR3010172,
MR3221852}
and their generalizations incorporating potentials
\cite{1405.6352,1506.02109}.
Another candidate is quasimap theory,
which is an intersection theory
on moduli spaces of maps
to the quotient stacks $[W/G]$.
A review of the latter theory,
with historical remarks and extensive references,
can be found in \cite{Ciocan-Fontanine--Kim}.
These two approaches should be related
by Hitchin--Kobayashi correspondence
for vortices \cite{MR1085139,MR1801657,1301.7052}.

When the gauge group is abelian,
quasimap theory as a mathematical theory
of A-twisted gauged linear sigma models
goes back to \cite{MR1336089}.
The relation with the Yukawa coupling of the mirror
is formulated as \emph{toric residue mirror conjecture}
in \cite{MR1988969,MR2019144}
and proved in
\cite{MR2104791,MR2099774,
MR2147350,MR2218757}.

Quasimap theory in the special case of projective hypersurfaces
is also studied in the insightful paper
\cite{MR1354600},
where a heuristic relation
with semi-infinite homologies of loop spaces
is discussed.
This eventually leads to Givental's proof
\cite{MR1408320}
of classical mirror symmetry
\cite{MR1115626}
for the quintic 3-fold.
This has been extended to toric complete intersections
in \cite{MR1653024}.

The correlation functions
of A-twisted gauged linear sigma models
in the cases when gauge groups are not necessarily abelian
are computed in \cite{MR3383085, MR3370259}
using supersymmetric localization of path integrals.
The result is given in terms of Jeffrey--Kirwan residues,
and reproduces the results of \cite{MR1336089}
in abelian cases.

The aim of this paper is twofold.
One is to give an expository account of quasimap theory
and its relation to other subjects
such as instantons and integrable systems.
The other is to formulate \pref{cj:main},
which states that the correlation function
defined in \pref{eq:correlator}
in terms of residues
coincides with
the generating function
of quasimap invariants
defined in \pref{eq:quasimap_correlator},
and prove it for Grassmannians
in \pref{sc:RMSG}.
This can be considered as a generalization
of toric residue mirror symmetry to Grassmannians.
%The other is to discuss toric residue mirror symmetry
%and its non-abelian generalization.
%We formulate a conjecture,
%which equates the correlation functions of
%A-twisted gauged linear sigma models
%computed in \cite{MR3383085, MR3370259}
%with quasimap invariants,
%and prove it for Grassmannians.
We also show in \pref{sc:TR}
that a slightly weakened version of
toric residue mirror conjecture follows
from Givental's mirror theorem.
Nothing else in this paper is new.

%is closely related to abelian/non-abelian correspondence
%\cite{MR2110629,
%MR2369087,
%MR2367022,
%MR2388562}.
%and it is an interesting problem
%to explore this relationship.
%We also discuss the relation
%with Vafa-Intriligator formulas
%\cite{MR1093562,MR1138873,MR1621570}
%and Bethe/gauge correspondence
%\cite{MR2570974}.

This paper is organized as follows:
In \pref{sc:correlator},
we recall the description
of correlation functions
of A-twisted gauged linear sigma models
given in \cite{MR3383085, MR3370259}.
In \pref{sc:QSPS},
we recall the definition of the quasimap spaces
$\Q(\bP^{n-1};d)$.
They are compactifications of the spaces of holomorphic maps
of degrees $d$ from $\bP^1$ to $\bP^{n-1}$, and
play an essential role
in Givental's homological geometry
\cite{MR1354600,MR1403947}.
%We also recall \pref{th:P-correlator},
%which gives the relation between the generating function
%of intersection numbers on $\cQ(\bP^{n-1};d)$
%and the quantum cohomology of $\bP^{n-1}$.
%
In \pref{sc:PCI},
we recall toric residue mirror symmetry
for Calabi--Yau complete intersections in projective spaces.
%following \cite{MR1988969,MR2019144}.
%
In \pref{sc:concave},
we discuss quasimap invariants of concave bundles.
In \pref{sc:CMSTH},
we recall classical mirror symmetry for toric hypersurfaces
proved in \cite{MR1653024}.
The exposition in \pref{sc:CMSTH} follows
\cite{MR3112512} closely.
In \pref{sc:QSTV},
we briefly recall the definition of quasimap spaces for toric varieties
due to \cite{MR1336089}.
In \pref{sc:TR},
we show that a slightly weakened version
of toric residue mirror conjecture for CY hypersurfaces
follows from classical mirror symmetry.
In \pref{sc:Martin},
we recall a theorem of Martin
which relates integration on a symplectic quotient by a compact Lie group
to that on the quotient by a maximal torus.
In \pref{sc:QS},
we recall the definition of quasimap spaces
to GIT quotients,
which are called \emph{quasimap graph spaces}
in \cite{MR3272909}.
The quasimap spaces
come with the universal $G$-bundle and
%the evaluation map at $p \in \bP^1$ and
the canonical virtual fundamental classes,
which allow us to define numerical invariants.
We formulate \pref{cj:main},
which states that correlation functions
of A-twisted gauged linear sigma models
given in \pref{eq:correlator}
are generating functions
of quasimap invariants.
There is a natural $\bGm$-action
on the quasimap graph space
coming from the $\bGm$-action on the domain curve.
There is a distinguished connected component
of the fixed locus of this action,
which is used to define the \emph{$I$-function}.
In \pref{sc:QSG},
the quasimap spaces and the $I$-functions for Grassmannians
are recalled from \cite{MR2110629}.
%
%In \pref{sc:AQG},
%we introduce the \emph{abelianized quasimap space}
%for Grassmannians,
%which allows us to relate quasimap invariants for Grassmannians
%with correlations function in \pref{eq:correlator}
%(or their equivariant versions in \pref{eq:correlator2}).
%using the Vafa-Intriligator formula.
%proved in \cite{MR1621570}.
%
In \pref{sc:RMSG},
we prove \pref{cj:main} for Grassmannians.
For this purpose,
we introduce \emph{abelianized quasimap spaces}
for Grassmannians,
which allows us to relate quasimap invariants for Grassmannians
with correlations function in \pref{eq:correlator}.
%discuss a generalization of toric residue mirror symmetry
%to complete intersection in Grassmannians.
%$\Gr(r,n) = \Mat(r,n) \GIT U(r)$
%defined by vector bundles
%associated with representations of $U(r)$.
%
In \pref{sc:B/g},
we discuss the relation between gauged linear sigma models
and Bethe ansatz
%in solvable lattice models
following \cite{MR2570974}.
In Sections
\ref{sc:instanton},
\ref{sc:monopole}, and
\ref{sc:vortex},
%In Sections \ref{sc:instanton}--\ref{sc:vortex},
we recall the relations of quasimaps
with instantons, monopoles, and vortices respectively.

\ \\
\emph{Acknowledgements}:
We thank Ionu\c t Ciocan-Fontanine, Hiroshi Iritani, and Makoto Miura
for valuable discussions.
We also thank the anonymous referees
for suggesting several improvements, and
Korea Institute for Advanced Study
for financial support and excellent research environment.
K.~U.~is supported by JSPS KAKENHI Grant Numbers
24740043,
15KT0105,
16K13743, and
16H03930.
B.~K.~is supported by the KIAS individual grant MG016403. %the NRF grant 2007-0093859.

\section{Correlation functions of A-twisted gauged linear sigma models}
 \label{sc:correlator}

\subsection{}

Let $G$ be a reductive algebraic group of rank $r$, and
$W$ be a representation of $G \times \bGm$.
The center of $G$ and its Lie algebra
will be denoted by $Z(G)$ and $\fz$.
Fix a maximal torus $T$ of $G$,
and let $\ft$ be its Lie algebra.
The set of roots, its subset of positive roots, and the Weyl group
will be denoted by $\Delta$, $\Delta_+$, and $\W \coloneqq N(T) / T$.
Let
$
 W = \bigoplus_{i=1}^N W_i
$
be the weight space decomposition of $W$
with respect to the action of $T \times \bGm$.
The weight of $W_i$ will be denoted by
$(\rho_i, r_i) \in \ft^\vee \oplus  \bZ$,
and $r_i$ will be called the \emph{R-charge}.
%(i.e., the weight of the $\bGm$-action).
%
If $W$ admits an action of another torus $H$
commuting with the action of $G \times \bGm$,
then
one can introduce the \emph{twisted mass} $\lambda \in \frakh$ in the theory,
which corresponds to the equivariant parameter for the $H$-action.
The $T \times \bGm \times H$-weight of $W_i$ will be denoted by
$(\rho_i, r_i, \nu_i) \in \frakt^\vee \oplus \bZ \oplus \frakh^\vee$.
We also introduce the \emph{complexified Fayet--Illiopoulos parameter}
$t' \in \fz^\vee  \otimes_{\bR} \bC$,
which corresponds to the complexified K\"ahler form
of the symplectic quotient.
Here, we save the unprimed symbol $t$
for the indeterminate
in the generating function
of quasimap invariants
(see \pref{eq:tt'1}, \pref{eq:tt'2}, \pref{eq:tt'4} and \pref{eq:tt'3}).

For $d \in \ft$
and  $t' \in \fz^\vee$,
the composition 
of the surjection
$
 \ft^\vee \twoheadrightarrow \fz^\vee
$
dual to the inclusion
$
 \fz \hookrightarrow \ft
$
and the evaluation
$
 \ft^\vee \times \ft \to \bC
% \bR
$
will be denoted by
$
 \inner{t'}{d}
$
or
$
t'(d).
$
%For a $T$-representation $R$, $\Delta (R)$ will denote 
%the set of weights of the representation. 

\subsection{}

For $d \in \ft$ and $x \in \ft$, let
\begin{align} \label{eq:Zd}
 Z_d(x) &\coloneqq \Zvec_d(x) \Zmat_d(x)
\end{align}
be the product of
\begin{align} \label{eq:Zvecd}
 \Zvec_d(x) \coloneqq \prod_{\alpha \in \Delta_+} (-1)^{\alpha(d) +1} \alpha^2 (x)
\end{align}
and
\begin{align} \label{eq:Zmatd}
 \Zmat_d(x) &\coloneqq
  \prod_{i=1}^N (\rho_i(x) + \nu_i (\lambda) )^{r_i-\rho_i (d)-1}.
\end{align}
Here the superscripts `vec' and `mat' stands
for the vector multiplet
and the matter chiral multiplet respectively.
According to \cite{MR3383085, MR3370259},
the correlation function
of a $\W$-invariant polynomial
$
 P(x) \in \bC[\ft]^\W
$
on a 2-sphere is given,
up to sign introduced by hand,
by
\begin{align} \label{eq:correlator}
 \cor{P(x)}
= \frac{1}{|\W|} \sum_{d \in \mathsf{P}^\vee} e^{\inner{t'}{d}}
   \JK_\fc(Z_d(x) P(x)).
\end{align}
Here $\mathsf{P}^\vee$ is the coweight lattice of $G$ and
$\JK_\frakc$
is the Jeffrey--Kirwan residue
defined in \cite[Section 2]{MR2104791}
(cf.~also \cite{MR1710758}).
The cone $\fc \subset \fz^\vee$ is the ample cone
of the GIT quotient determined by the Fayet--Iliopoulos parameter $\eta$.

\subsection{}

One can introduce a variable $\z$
associated with the background value
of an auxiliary gauge field in the gravity multiplet.
This corresponds to the equivariant parameter
for the $\bGm$-action on the domain curve.
This turns \eqref{eq:Zd} into
the product
\begin{align} \label{eq:Zd2}
 Z_d(x;\z) &\coloneqq \Zvec_d(x;\z) \Zmat_d(x;\z)
\end{align}
of
\begin{align} \label{eq:Zvecd2}
 \Zvec_d(x;\z) &\coloneqq
  \prod_{\alpha \in \Delta_+} (-1)^{\alpha(d) +1} \alpha(x) \alpha(x + d \z)
\end{align}
and
\begin{align} \label{eq:Zmatd2}
 \Zmat_d(x;\z) &\coloneqq
  \prod_{i=1}^N
  \frac{\prod_{l=-\infty}^{-1}
   \lb \rho_i(x) + \nu_i(\lambda) - \lb l + \frac{r_i}{2} \rb \z \rb
  }
  {\prod_{l=-\infty}^{\rho_i(d)-r_i}
   \lb \rho_i(x) + \nu_i(\lambda) - \lb l + \frac{r_i}{2} \rb \z \rb
  },
\end{align}
\begin{comment}
\begin{align} \label{eq:Zvecd2'}
 \Zvec_d(x;\z) &\coloneqq
  \prod_{\alpha \in \Delta_+} (-1)^{\alpha(d) +1} \alpha \lb x-\frac{\z}{2} d  \rb \alpha \lb x+\frac{\z}{2} d  \rb
\end{align}
and
\begin{align} \label{eq:Zmatd2'}
 \Zmat_d(x;\z) &\coloneqq
  \prod_{i=1}^n \prod_{\rho \in \Delta(R_i)}
  \prod_{l=-\frac{\abs{\rho(d)-r_i+1}-1}{2}}^{\frac{\abs{\rho(d)-r_i+1}-1}{2}}
   \lb \rho(x) - \lambda_i - l \z \rb^{\sgn(r_i-\rho(d)-1)},
\end{align}
\end{comment}
and the correlation function of
$
 P(x) \in \bC[\ft]^\W
$
is given by
\begin{align} \label{eq:correlator2}
 \cor{P(x)}^{H \times \bGm}
= \frac{1}{|\W|} \sum_{d \in \mathsf{P}^\vee} e^{\inner{t'}{d}}
   \JK_\fc (Z_d(x;\z) P(x)).
\end{align}

\subsection{}

Another quantity of interest is the \emph{effective twisted superpotential}
on the Coulomb branch,
or the \emph{effective potential} for short.
It is defined as the sum
\begin{align} \label{eq:Weff}
 \Weff(x;t') &\coloneqq \WFI(x;t') + \Wvec(x) + \Wmat(x)
\end{align}
of the Fayet-Illiopoulos term
\begin{align} \label{eq:WFI}
 \WFI(x;t') &\coloneqq \inner{t'}{x},
\end{align}
the vector multiplet term
\begin{align} \label{eq:Wvec}
 \Wvec(x) &\coloneqq - \pi \sqrt{-1} \sum_{\alpha \in \Delta^+} \alpha(x),
\end{align}
and the matter term
\begin{align} \label{eq:Wmat}
 \Wmat(x) &\coloneqq - \sum_{i=1}^N
  \lb \rho_i(x) + \nu_i(\lambda) \rb \lb \log \lb \rho_i(x) + \nu_i(\lambda) \rb - 1 \rb.
\end{align}

\section{Quasimap spaces for projective spaces}
 \label{sc:QSPS}

\subsection{} \label{QmaptoP}

A holomorphic map
$
 u \colon \bP^1 \to \bP^{n-1}
$
of degree $d$
is given by a collection
$\lb u_i(z_1,z_2) \rb_{i=1}^n$
of $n$ homogeneous polynomials of degree $d$
satisfying the following condition:
\begin{align} \label{eq:map_condition}
 \alitem{
There exists no $(z_1,z_2) \in \bA^2 \setminus \{ 0 \}$
such that
%$u_i(z_1,z_2) = 0$ for all $i \in \{ 1, \ldots, n \}$.
$u(z_1,z_2) = 0 \in \bA^n$.
}
\end{align}
Two collections $\lb u_i(z_1,z_2) \rb_{i=1}^n$
and $\lb u'_i(z_1,z_2) \rb_{i=1}^n$ define
the same map if and only if there exists $\alpha \in \bGm$
such that $u_i(z_1,z_2) = \alpha u'_i(z_1,z_2)$
for all $i \in \{ 1, \ldots, n\}$.
It follows that
the space
\begin{align}
 \cM(\bP^{n-1}; d)
  &\coloneqq \lc u \colon \bP^1 \to \bP^{n-1} \relmid \deg u = d \rc
\end{align}
of holomorphic maps of degree $d$
from $\bP^1$ to $\bP^{n-1}$
can be compactified
to the projective space
of dimension $n(d+1)-1$,
whose homogeneous coordinate is given by
the coefficients $\lb a_{ij} \rb_{i,j}$
of the collection $(u_i(z_1,z_2))_{i=1}^n$
of homogeneous polynomials of degree $d$;
\begin{align}
 u_i(z_1,z_2) = \sum_{j=0}^d a_{ij} z_1^j z_2^{d-j}, \qquad
 i=1, \ldots, n.
\end{align}
%where $i$ runs from 1 to $n$
%and $j$ runs from 1 to $d$.
This compacification is called the \emph{quasimap space}
and denoted by $\Q(\bP^{n-1}; d)$.
%\begin{align}
% \cQ(\bP^{n-1}; d)
%  &\coloneqq \lc \ld \varphi_1(z, w) : \cdots : \varphi_n(z, w) \rd \in \bP^{n(d+1)-1} \relmid
%   \varphi_i(z, w) = \sum_{j=0}^d a_{ij} z^i w^{d-j} \rc
%   \label{eq:quasimap2}
%\end{align}
%of holomorphic maps
%from $\bP^1$ to $\bP^{n-1}$
%is a 
%admits a simple compactification
%The space on the right hand side of \pref{eq:quasimap1}
%consists of collections of $n$ homogeneous polynomials
%of degree $d$
%up to simultaneous scalar multiplication.
%It is a projective space of dimension $n(d+1)-1$
%whose homogeneous coordinate is given by
%$\lb a_{ij} \rb_{i,j}$,
%where $i$ runs from 1 to $n$
%and $j$ runs from 1 to $d$.
An element of the quasimap space
is called a \emph{quasimap}.

\subsection{}

A point $[z_1:z_2] \in \bP^1$ is a \emph{base point}
(or \emph{singularity})
of a quasimap $u$
if $u(z_1,z_2)=0$.
A quasimap is a genuine map
outside of the base locus.
If the degree of the base locus is $d'$,
then a quasimap can be considered as a genuine map
of degree $d-d'$.
However, it is more convenient to think of a quasimap
as a morphism to the quotient stack
$[\bA^n/\bGm]$.
By definition,
a morphism from $\bP^1$ to $[\bA^n/\bGm]$
is a principal $\bGm$-bundle $P$ over $\bP^1$
and a $\bGm$-equivariant morphism $\utilde \colon P \to \bA^n$.
It is a quasimap if the generic point of $P$ is mapped
to the semi-stable locus $\bA^n \setminus \{ 0 \}$.

\subsection{}

Let $x \in H^2(\Q(\bP^{n-1};d); \bZ)$ be the ample generator
of the cohomology ring of $\Q(\bP^{n-1};d) \cong \bP^{n(d+1)-1}$,
so that
\begin{align}
 H^*(\Q(\bP^{n-1};d); \bZ) \cong \bZ[x]/ \lb x^{n(d+1)} \rb.
\end{align}
Given a polynomial $P(x) \in \bC[x]$,
we set
\begin{align}
 \la P(x) \ra_{\bP^{n-1}}
  \coloneqq \sum_{d=0}^\infty q^d \la P(x) \ra_{\bP^{n-1},d}
  \in \bC \db[ q \db],
\end{align}
where
\begin{align}
 \la P(x) \ra_{\bP^{n-1},d}
  \coloneqq \int_{\Q(\bP^{n-1};d)} P(x)
\end{align}
is the integration over the quasimap space.
It follows from
\begin{align}
 \la x^k \ra_{\bP^{n-1},d} =
\begin{cases}
 1 & k = n(d+1)-1, \\
 0 & \text{otherwise}
\end{cases}
\end{align}
that
\begin{align} \label{eq:P-correlator1}
 \la x^k \ra_{\bP^{n-1}} =
\begin{cases}
 q^d & k = n(d+1)-1 \text{ for some } d \in \bZ^{\ge 0}, \\
 0 & \text{otherwise}.
\end{cases}
\end{align}

\subsection{}

If we set $G \coloneqq \bGm$ and $W \coloneqq \bC^n$
with the action
$
 G \times \bGm \ni (\alpha, \beta) \colon (w_1,\ldots,w_n)
  \mapsto (\alpha w_1, \ldots, \alpha w_n),
$
then we have
$
 \Zvec_d(x) = 1
$
and
$
 \Zmat_d(x) = \lb x^{-d-1} \rb^n,
$
so that \pref{eq:correlator} gives
the same result as \pref{eq:P-correlator1}
under the identification
\begin{align} \label{eq:tt'1}
 q = e^{t'}.
\end{align}

\subsection{}

The \emph{small quantum cohomology} of $\bP^{n-1}$ is
the free $\bC \db[ q \db]$-module
\begin{align}
 \QH \lb \bP^{n-1} \rb
  \coloneqq H^* \lb \bP^{n-1}; \bC \db[ q \db] \rb
\end{align}
equipped with multiplication given by
\begin{align} \label{eq:GWPn-1}
 x^i \circ x^j
  \coloneqq \sum_{k=0}^n \sum_{d=0}^\infty q^d
   \la I_{0,3,d} \ra \! (x^i, x^j, x^k)
   \, x^{n-k-1}.
\end{align}
Here
\begin{align}
 \la I_{0,3,d} \ra \! (a, b, c)
  \coloneqq \int_{\ld \cMbar_{0,3} \lb \bP^{n-1};d \rb \rd^\virt}
    \ev_1^* a \cup \ev_2^* b \cup \ev_3^* c
\end{align}
is the 3-point Gromov-Witten invariant.
It is an associative commutative deformation
of the classical cohomology ring;
\begin{align}
 \left. \QH \lb \bP^{n-1} \rb \right/ \lb q \rb
  \cong H^* \lb \bP^{n-1}; \bC \rb.
\end{align}
Since the virtual dimension of the moduli space of stable maps is given by
\begin{align}
 \vdim \cMbar_{g,k}(X;d)
  = (1-g) (\dim X-3) + \la c_1(X),d \ra + k
\end{align}
in general,
one has
\begin{align}
 \vdim \cMbar_{0,3} \lb \bP^{n-1};d \rb
  = nd+n-1.
\end{align}
The 3-point Gromov-Witten invariant
in \pref{eq:GWPn-1}
is non-zero only if
\begin{align} \label{eq:Pn-1-degree}
 \vdim \cMbar_{0,3} \lb \bP^{n-1};d \rb
  = i+j+k.
\end{align}
Since $0 \le i,j,k \le n-1$,
one has \pref{eq:Pn-1-degree}
only if $d=0$, $i+j+k=n-1$ or $d=1$, $i+j+k=2n-1$.
This shows that $x^i \circ x^j = x^{i+j}$ for $i+j \le n-1$.
%One has
%\begin{align}
% \int_{\ld \cMbar_{0,3} \lb \bP^{n-1};d \rb \rd^\virt}
%    \ev_1^* \lb x^{n-1} \rb \cup \ev_2^* \lb x^{n-1} \rb \cup \ev_3^* x
%  = 1
%\end{align}
Since
there is a unique line passing through two points on $\bP^{n-1}$
in general position,
and this line intersects a hyperplane at one point,
one has
$
 x \circ x^{n-1} = q.
$
Hence the ring structure of the quantum cohomology of $\bP^{n-1}$
is given by
\begin{align} \label{eq:QHP}
 \QH \lb \bP^{n-1} \rb \cong \left. \lb \bC \db[ q \db] \rb [x] \right/ \lb x^n-q \rb.
\end{align}
We write the ring homomorphism
$
 \bC[x] \to \QH(\bP^{n-1})
$
sending $x$ to $x$ as
$
 P(x) \mapsto \mathring{P}(x).
$

\begin{theorem} \label{th:P-correlator}
For any $P(x) \in \bC[x]$,
one has
\begin{align} \label{eq:P-correlator}
 \la P(x) \ra_{\bP^{n-1}}
  = \int_{\bP^{n-1}} \mathring{P}(x).
\end{align}
\end{theorem}

\begin{proof}
Since both sides of \pref{eq:P-correlator} are linear in $P(x) \in \bC[x]$,
it suffices to show
\begin{align}
 \la x^k \ra_{\bP^{n-1}} = \int_{\bP^{n-1}} x^{\circ k}
\end{align}
for any $k \in \bN$,
which is obvious from \pref{eq:P-correlator1}
and \pref{eq:QHP}.
\end{proof}

\pref{th:P-correlator} is equivalent
to the Vafa-Intriligator formula
\cite{MR1093562,MR1138873}:

\begin{corollary}[Vafa-Intriligator formula for projective spaces]
 \label{cr:IV1}
For any $P(x) \in \bC[x]$,
one has
\begin{align} \label{eq:Pn-1_IV}
 \int_{\bP^{n-1}} \mathring{P}(x)
  = \frac{1}{n} \sum_{\lambda^n = q} \frac{P(\lambda)}{\lambda^{n-1}},
\end{align}
where the sum is over $\lambda \in \bC \db[ q^{1/n} \db]$
satisfying $\lambda^n = q$.
\end{corollary}

\begin{proof}
Since the integration over the projective space
can be written by residue as
\begin{align} \label{eq:Pr_int}
 \int_{\bP^{r-1}} x^k
  = \delta_{r-1, k}
  = \Res \frac{x^k dx}{x^r},
\end{align}
one has
\begin{align}
 \int_{\bP^{n-1}} \mathring{P}(x)
  &= \la P(x) \ra_{\bP^{n-1}} \\
  &= \sum_{d=0}^\infty q^d \int_{\Q(\bP^{n-1};d)} P(x) \\
  &= \sum_{d=0}^\infty q^d \Res \frac{P(x) dx}{x^{n(d+1)}} \\
  &= \Res \frac{x^{-n} P(x)}{1-q x^{-n}} \\
  &= \Res \frac{P(x)}{x^n-q} \\
  &= \frac{1}{n} \sum_{\lambda^n = q} \frac{P(\lambda)}{\lambda^{n-1}},
\end{align}
and \pref{eq:Pn-1_IV} is proved.
\end{proof}

\subsection{}

The projective space $\bP^{n-1}$ has a natural action of $\GL_n$,
which restricts to the action
of the diagonal maximal torus $H$.
The \emph{equivariant cohomology} is defined as the ordinary cohomology
$
 H^*_H(\bP^{n-1}) \coloneqq
  H^*(\bP^{n-1}_H)
$
of the \emph{Borel construction}
$
 \bP^{n-1}_H \coloneqq \bP^{n-1} \times_H EH,
$
where $EH$ is the product
of $n$ copies of the total space of the tautological bundle
$\cO_{\bP^{\infty}}(-1)$ over $B\bGm = \bP^{\infty}$.
It follows that $\bP^{n-1}_H$ is the projectivization $\bP(\cE)$
of the vector bundle $\cE \coloneqq \bigoplus_{i=1}^n \pi_i^* \cO_{\bP^{\infty}}(-1)$
of rank $n$ over $(\bP^\infty)^n$.
A standard result
on the cohomology of a projective bundle
(see e.g. \cite[page 606]{MR507725})
shows that
$H^*(\bP^{n-1}_H)$
is generated over
$
 H_H^*(\mathrm{pt})
  = H^* \lb (\bP^\infty)^n \rb
  \cong \bC[\lambda_1,\ldots,\lambda_n]
$
by
$
 x \coloneqq - c_1(\cO_{\bP(\cE)}(-1))
$
with one relation
\begin{align} \label{eq:equiv_coh1}
 (-x)^n - c_1(\cE) (-x)^{n-1} + c_2(\cE) (-x)^{n-2} + \cdots + (-1)^n c_n(\cE) = 0.
\end{align}
Since $c_i(\cE) = (-1)^i \sigma_i(\lambda_1, \ldots, \lambda_n)$,
one obtains
\begin{align}
 H_H^*(\bP^{n-1})
  \cong \bC[x,\lambda_1,\ldots,\lambda_n] \left/
   \prod_{i=1}^n (x-\lambda_i) \right. .
\end{align}
The $H$-fixed locus $(\bP^{n-1})^H$ consists of $n$ points
$\{ p_i \}_{i=1}$,
where $p_i$ is the point $[z_1:\cdots:z_n] \in \bP^{n-1}$
with $z_i = 1$ and $z_j = 0$ for $i \ne j$.
Since the tautological bundle $\cO_{\bP(\cE)}(-1)$
restricts to $\pi_i^{*} \cO_{\bP^\infty}(-1)$
on $(p_i)_T = (\bP^\infty)^n$,
one has
\begin{align}
 \iota_i^* x = \lambda_i.
\end{align}
The push-forward
\begin{align}
 \int_{\bP^{n-1}}^H \colon H_H^*(\bP^{n-1})
  \to H_H^*(\mathrm{pt})
  \cong \bC[\lambda_1,\ldots,\lambda_n]
\end{align}
along the natural map
$
 (\bP^{n-1})_H \to (\mathrm{pt})_H \cong BH
$
is called the \emph{equivariant integration}.
The localization theorem \cite{MR721448} shows
\begin{align}
 \int_{\bP^{n-1}}^H P(x)
%  \coloneqq \int_{X_H} P(x)
  &= \sum_{i=1}^n \frac{\iota_i^* P(x)}{\Eul^H(N_{p_i/\bP^{n-1}})} \\
  &= \sum_{i=1}^n \frac{P(\lambda_i)}{\prod_{j \ne i} (\lambda_i - \lambda_j)}
   \nonumber \\
  &= \Res \frac{P(x) dx}{\prod_{i=1}^n (x-\lambda_i)}
   \nonumber
\end{align}
for any $P(x) \in H^*_H(\bP^{n-1})$.

\subsection{}

The quasimap space $\Q(\bP^{n-1};d)$
has a natural action of $H \times \bGm$
given by
\begin{align}
 H \times \bGm \ni (\alpha_1,\ldots,\alpha_n,\beta) \colon
  \lb u_i(z_1,z_2) \rb_{i=1}^n
   \mapsto \lb \alpha_i u_i(z_1,\beta z_2) \rb_{i=1}^n.
\end{align}
The equivariant cohomology of $\Q(\bP^{n-1};d)$
with respect to this torus action
is given by
\begin{align}
 H_{H \times \bGm}^*(\Q(\bP^{n-1};d);\bC) \cong
 \bC[x,\lambda_1,\ldots,\lambda_n,\z]
  \left/ \lb \prod_{i=1}^n \prod_{j=0}^d (x-\lambda_i-j\z) \rb \right. .
\end{align}
The $H \times \bGm$-equivariant integration
\begin{align}
 \la - \ra_{\bP^{n-1},d}^{H \times \bGm}
  \colon H_{H \times \bGm}^*(\Q(\bP^{n-1};d);\bC) \to H^*(B(H \times \bGm);\bC)
\end{align}
is given by
\begin{align} \label{eq:HU-equiv_correlator1}
 \la P(x) \ra_{\bP^{n-1},d}^{H \times \bGm}
  &= \Res \frac{P(x) dx}{\prod_{i=1}^n \prod_{j=0}^d (x-\lambda_i-jz)} \\
  &= \sum_{i=1}^n \sum_{j=0}^d \frac{P(\lambda_i+j\z)}
  {\prod_{(k,l) \ne (i,j)} ((\lambda_i+j\z)-(\lambda_k+l\z))}.
\end{align}
The $H \times \bGm$-equivariant correlator is given by
\begin{align} \label{eq:HU-equiv_correlator}
 \la P(x) \ra_{\bP^{n-1}}^{H \times \bGm}
  &\coloneqq \sum_{d=0}^\infty q^d \la P(x) \ra_{\bP^{n-1},d}^{H \times \bGm}.
\end{align}
The $H$-equivariant correlator
$
 \la P(x) \ra_{\bP^{n-1}}^H
$
and the $\bGm$-equivariant correlator
$
 \la P(x) \ra_{\bP^{n-1}}^{\bGm}
$
are obtained by
setting $\z = 0$
and $\bslambda = (\lambda_1,\ldots,\lambda_n)=0$
respectively.
%in \pref{eq:HU-equiv_correlator};
%\begin{align}
% \la P(x) \ra_{\bP^{n-1}}^H
% &= \sum_{d=0}^\infty q^d \Res \frac{P(x) dx}{\prod_{i=1}^n (x-\lambda_i)^{d+1}} \\
% &= \Res \frac{P(x) dx}{\prod_{i=1}^n (x-\lambda_i) (1-q(x-\lambda_i)^{-1})} \\
% &= \Res \frac{P(x) dx}{\prod_{i=1}^n (x-\lambda_i-q)}.
%\end{align}

\subsection{}

The fixed point of the $\bGm$-action on $\Q(\bP^{n-1};d)$
is the disjoint union
\begin{align}
 \Q(\bP^{n-1};d)^{\bGm}
  = \coprod_{i=0}^d \Q(\bP^{n-1};d)^{\bGm}_i
\end{align}
of $d+1$ connected components
\begin{align} \label{eq:P_conn_comp_i}
 \Q(\bP^{n-1};d)^{\bGm}_i
  \coloneqq \lc [a_1 z_1^i z_2^{d-i}, \ldots, a_n z_1^i z_2^{d-i}] \in \Q(\bP^{n-1};d)
   \relmid [a_1,\ldots,a_n] \in \bP^{n-1} \rc.
\end{align}
Each of these connected components is isomorphic to $\bP^{n-1}$,
and the base locus is $i 0 + (d-i) \infty$.
The connected component
$
 \Q(\bP^{n-1};d)^{\bGm}_0
$
will be denoted by $\Qone(\bP^{n-1};d)$.
There is a natural map
$\ev \colon \Qone(\bP^{n-1};d) \to \bP^{n-1}$
called the \emph{evaluation map},
and one has
\begin{align}
 \Q(\bP^{n-1};d)^{\bGm}
  \cong \coprod_{d_1+d_2=d} \Qone(\bP^{n-1};d_1)
   \times_{\bP^{n-1}} \Qone(\bP^{n-1};d_2).
\end{align}
The normal bundle of
$\Qone(\bP^{n-1};d)$
in $\Q(\bP^{n-1};d)$
is given by $\cO_{\bP^{n-1}}(1)^{\oplus nd}$,
whose equivariant Euler class is given by
\begin{align}
 \Eul^{H \times \bGm} \lb N_{\Qone(\bP^{n-1};d) / \Q(\bP^{n-1};d)} \rb
  = \prod_{i=1}^n \prod_{l=1}^d \lb x - \lambda_i + l \z \rb.
\end{align}
The \emph{equivariant $I$-function} is defined by
\begin{align} \label{eq:equiv_I_P}
 I_{\bP^{n-1}}^H(t;\z)
  &\coloneqq e^{t x/z} \sum_{d=0}^\infty e^{dt} I^H_d
\end{align}
where
\begin{align}
 I^H_d(\z)
  &\coloneqq \ev_* \lb \frac{1}
   {\Eul^{H \times \bGm} \lb N_{\Qone(\bP^{n-1};d) / \Q(\bP^{n-1};d)} \rb} \rb \\
  &= \frac{1}
  {\prod_{i=1}^n \prod_{l=1}^d \lb x - \lambda_i + l \z \rb}.
\end{align}
The non-equivariant $I$-function is defined similarly,
and given by setting $\bslambda=0$ in \pref{eq:equiv_I_P};
\begin{align}
 I_{\bP^{n-1}}(t;\z)
  &\coloneqq e^{t x/z} \sum_{d=0}^\infty \frac{e^{dt}}{\prod_{l=1}^d \lb x + l \z \rb^n}.
\end{align}

\subsection{}

Let
$
 \lb \bC \db[ e^t \db] \rb [t]
$
be the polynomial ring in $t$
with the ring $\bC \db[ e^t \db]$
of formal power series in $e^t$ as a coefficient.
The equivariant $I$-function in \pref{eq:equiv_I_P} is an element of
$
 H_H^* \lb \bP^{n-1}; \bC \rb \otimes_{\bC} \lb \bC \db[ e^t \db] \rb [t],
$
and the variable $t$ is related to the variable
$q$ appearing in the correlator by
\begin{align} \label{eq:tt'2}
 q = e^t.
\end{align}
The equivariant $I$-function can also be considered
as a $H_H^*\lb \bP^{n-1}; \bC \rb$-valued analytic function,
which is multi-valued as a function of $q$ and
single-valued as a function of $t = \log q$.

\subsection{}
There is a $\bGm$-equivariant evaluation map
$
 \ev_0 \colon \Q (\bP^{n-1};d) \to \ld \bC^n / \bGm \rd
$
at the point $0 \in \bP^1$.
By abuse of notation,
we also let $x$ denote the $\bGm$-equivariant Euler class
of the line bundle
$
 \ev_0^*(\cO_{[\bC^n/\bGm]}(1)).
$
Here $\cO_{[\bC^n/\bGm]}(1)$ is the line bundle
$[(\bC^n \times \bC) / \bGm]$
on the quotient stack
with weights $((1, \ldots, 1), 1)$.

Let
$
 \iota_i \colon \Q(\bP^{n-1};d)_i^{\bGm} \to \Q(\bP^{n-1};d)
$
be the inclusion of the $i$-th connected component
\pref{eq:P_conn_comp_i}.
Since $\iota_i^*(x) = x+i \z$ (under the identification $\Q(\bP^{n-1};d)_i^{\bGm}  = \bP ^{n-1}$) and
\begin{align}
 \frac{1}{\Eul^{\bGm} \lb N_{\Q(\bP^{n-1};d)_i^{\bGm}/\Q(\bP^{n-1};d)} \rb}
  = I_i(\z) \cup I_{d-i}(-\z),
\end{align}
localization with respect to the $\bGm$-action
shows that
\begin{align} \label{eqn:decom Phi}
 \sum_{d=0}^\infty
  e^{d \tau} \la e^{(t-\tau)x/\z} \ra_{\bP^{n-1},d}^{\bGm}
 &= \sum_{d=0}^\infty e^{d \tau} \sum_{i=0}^d
 \int_{\Q(\bP^{n-1};d)^{\bGm}_i}
  \frac{\iota_i^* \lb e^{(t-\tau)x/\z} \rb}
  {\Eul^{\bGm} \lb N_{\Q(\bP^{n-1};d)_i^{\bGm}/\Q(\bP^{n-1};d)} \rb}  \nonumber  \\
 &= \sum_{d=0}^\infty e^{d \tau} \sum_{i=0}^d
 \int_{\bP^{n-1}}
  e^{(t-\tau)(x+i\z)/\z} \cup
  I_i(\z) \cup I_{d-i}(-\z)  \nonumber \\
 &=
 \sum_{d=0}^\infty \sum_{i=0}^d
 \int_{\bP^{n-1}}
  e^{t x /\z} e^{ti} I_i(\z) \cup
  e^{-\tau x/\z} e^{(d-i)\tau} I_{d-i}(-\z)     \nonumber \\
 &= \int_{\bP^{n-1}} I_{\bP^{n-1}}(t;\z) \cup I_{\bP^{n-1}}(\tau;-\z). 
\end{align}
The factorization of the $H \times \bGm$-equivariant correlator
is proved similarly as
\begin{align*}
% \la e^{(t-t')x/\z} \ra_{\bP^{n-1}}^{H \times \bGm}
 \sum_{d=0}^\infty & e^{d \tau} \la e^{(t-\tau)x/\z} \ra_{\bP^{n-1},d}^{H \times \bGm}
  \label{eq:factorization3} \\
  &= \sum_{d=0}^\infty
   \Res \frac{e^{d \tau} e^{(t-\tau)x/\z} dx}{\prod_{i=1}^n \prod_{l=0}^d (x-\lambda_i-l\z)} \\
  &= \sum_{d=0}^\infty \sum_{m=0}^d \sum_{j=1}^n
   \Res_{x=\lambda_j+m\z} \frac{e^{d \tau} e^{(t-\tau)x/\z} dx}
    {\prod_{i=1}^n \prod_{l=0}^d (x-\lambda_i-l\z)} \\
  &= \sum_{d=0}^\infty \sum_{m=0}^d \sum_{j=1}^n
   \Res_{x=\lambda_j} \frac{e^{d \tau} e^{(t-\tau) x/\z} e^{(t-\tau) m} dx}
    {\prod_{i=1}^n \prod_{l=0}^d (x-\lambda_i-(l-m)\z)} \\
%  &= \sum_{d=0}^\infty \sum_{m=0}^d \sum_{j=1}^n
%   \Res_{x=\lambda_j} \frac{e^{(t-\tau) x/\z} e^{mt} e^{(d-m)\tau} dx}
%    {\prod_{i=1}^n \prod_{l=0}^d (x-\lambda_i+(m-l)\z)} \\
  &= \sum_{d=0}^\infty \sum_{m=0}^d \sum_{j=1}^n
   \Res_{x=\lambda_j}
    \frac{e^{tx/\z} e^{mt}}{\prod_{i=1}^n \prod_{l=1}^m (x-\lambda_i+l\z)}
    \frac{e^{- \tau x/\z} e^{(d-m) \tau}}{\prod_{i=1}^n \prod_{l=1}^{d-m} (x-\lambda_i-l\z)}
    \frac{dx}{\prod_{i=1}^n (x-\lambda_i)} \\
  &= \sum_{d=0}^\infty \sum_{d'=0}^\infty \sum_{j=1}^n
   \Res_{x=\lambda_j}
    \frac{e^{tx/\z} e^{dt}}{\prod_{i=1}^n \prod_{l=1}^d (x-\lambda_i+l\z)}
    \frac{e^{-\tau x/\z} e^{d'\tau}}{\prod_{i=1}^n \prod_{l=1}^{d'} (x-\lambda_i-l\z)}
    \frac{dx}{\prod_{i=1}^n (x-\lambda_i)} \\
%  &= \sum_{j=1}^n \Res_{x=\lambda_j}
%    J_{\bP^{n-1}}^H(t;\z)
%    J_{\bP^{n-1}}^H(\tau;-\z)
%    \frac{dx}{\prod_{i=1}^n (x-\lambda_i)} \\
  &= \int_{\bP^{n-1}}^H I^H_{\bP^{n-1}}(t;\z) \cup I^H_{\bP^{n-1}}(\tau;-\z).
\end{align*}
This can also be regarded as a purely combinatorial proof.

%\subsection{}
%
%Let $\Q(\bP^{n-1};d)_0$ be the open subspace of $\Q(\bP^{n-1};d)$
%consisting of quasimaps
%whose base locus does not contain $0 \in \bP^1$.
%Since the $\bGm$-fixed locus of $\Q(\bP^{n-1};d)_0$
%coincides with $\Q(\bP^{n-1};d)^{\bGm}_0$,
%which is compact,
%the equivariant push-forward of
%$
% 1 \in H_{H \times \bGm}^0(\Q(\bP^{n-1};d)_0;\bZ)
%$
%along the evaluation map
%$
% \ev \colon \Q(\bP^{n-1};d)_0 \to \bP^{n-1}
%$
%makes sense,
%and coincides with $I_d$.

\subsection{}

Let
$
 \ev \colon \cMbar_{0,1}(\bP^{n-1};d) \to \bP^{n-1}
$
be the evaluation map
from the moduli space of stable maps
of genus 0 and degree $d$ with 1 marked point, and
$\psi$ be the first Chern class of the line bundle
over $\cMbar_{0,1}(\bP^{n-1};d)$
whose fiber at a stable map $\varphi \colon (C, x) \to \bP^{n-1}$
is the cotangent line $T_x^* C$ at the marked point.
The \emph{equivariant $J$-function}
\cite{MR1408320}
is a $H^*(\bP^{n-1};\bC)$-valued hypergeometric series
given by
\begin{align}
 J^H_{\bP^{n-1}}(t;\z)
  \coloneqq e^{t x/\z} \sum_{d=0}^\infty e^{dt} J_d
\end{align}
where
\begin{align}
 J_d \coloneqq \ev_* \lb \frac{1}{\z(\z-\psi)} \rb.
\end{align}

\subsection{}

The \emph{graph space}
is defined by
$
 G(\bP^{n-1};d) \coloneqq \cMbar_{0,0}(\bP^{n-1} \times \bP^1;(d,1)).
$
The source of any map
$\varphi \colon C \to \bP^{n-1} \times \bP^1$
in
$
% \cMbar_{0,0}(\bP^{n-1} \times \bP^1;(d,1))
 G(\bP^{n-1};d)
$
has a distinguished irreducible component
$C_1$
which maps isomorphically to $\bP^1$.
Let 
$
% \cMbar_{0,0}(\bP^{n-1} \times \bP^1;(d,1))_0
 G(\bP^{n-1};d)_0
$
be the open subspace of
$
% \cMbar_{0,0}(\bP^{n-1} \times \bP^1;(d,1))
 G(\bP^{n-1};d)
$
consisting of stable maps
without irreducible components
mapping constantly to $0 \in \bP^1$.
There is a map $\ev_0 \colon G(\bP^{n-1};d)_0 \to \bP^{n-1}$
sending $\varphi \colon C \to \bP^{n-1} \times \bP^1$ to
$
 \pr_1 \circ \varphi \lb \lb \pr_2 \circ \varphi \rb^{-1}(0) \rb.
$
The fixed locus of the natural $\bGm$-action on
$
% \cMbar_{0,0}(\bP^{n-1} \times \bP^1;(d,1))_0,
 G(\bP^{n-1};d)_0
$
can be identified with
$
 \cMbar_{0,1}(\bP^{n-1};d).
$
Since the natural morphism
$
 G(\bP^{n-1};d)_0 \to \Q(\bP^{n-1};d)_0
$
is a $\bGm$-equivariant birational morphism
which commutes with the evaluation maps,
the push-forwards $J_d$ and $I_d$
of 1 by
$
 \ev_0 \colon G(\bP^{n-1};d)_0 \to \bP^{n-1}
$
and
$
 \ev \colon \Q(\bP^{n-1};d)_0 \to \bP^{n-1}
$
are equal,
and hence
$
 I_{\bP^{n-1}}(t;\z) = J_{\bP^{n-1}}(t;\z).
$

\subsection{}

The effective potential \pref{eq:Weff} is given by
\begin{align}
 \Weff(x;t) = t x - \sum_{i=1}^n
  \lb x - \lambda_i \rb \lb \log \lb x - \lambda_i \rb - 1 \rb.
\end{align}
One can easily see
\begin{align}
 \frac{\partial \Weff}{\partial x}
  &= t - \sum_{i=1}^n \log(x-\lambda_i), \\
 e^{\partial_x \Weff}
  &= q \prod_{i=1}^n (x-\lambda_i)^{-1},
\end{align}
so that
\begin{align}
 \la P(x) \ra_{\bP^{n-1}}^H
 &= \Res \frac{P(x) dx}{\prod_{i=1}^n (x-\lambda_i)
  (1-e^{\partial_x \Weff})}.
\end{align}
Note that the equation
\begin{align}
 e^{\partial_x \Weff} = 1
\end{align}
gives the relation
\begin{align}
 \prod_{i=1}^n (x-\lambda_i) = q
\end{align}
in the equivariant quantum cohomology of $\bP^{n-1}$.

\section{Projective complete intersections}
 \label{sc:PCI}

\subsection{}

Let
$
 f_1(w_1, \ldots, w_n), \ldots, f_r(w_1, \ldots, w_n)
  \in \bC[w_1, \ldots, w_n]
$
be homogeneous polynomials
of degrees $l_1, \ldots, l_r$
satisfying the Calabi--Yau condition
\begin{align}
 l_1 + \cdots + l_r = n.
\end{align}
Assume that $f_1, \ldots, f_r$ are sufficiently general
so that
\begin{align}
 Y \coloneqq \lc [w_1,\cdots,w_n] \in \bP^{n-1} \relmid
  f_1(w_1,\ldots,w_n)=\cdots=f_r(w_1,\ldots,w_n)=0 \rc
\end{align}
is a smooth complete intersection of dimension $n-r-1$,
whose Poincar\'e dual is
\begin{align}
 v \coloneqq \prod_{i=1}^r (l_i x).
\end{align}
Define the quasimap space
$\Q(Y;d)$
as the subset of $\Q(\bP^{n-1};d)$
consisting of $[\varphi_1(z_1,z_2),\ldots,\varphi_n(z_1,z_2)]$
satisfying
\begin{align}
\alitem{$f_i(\varphi_1(z_1,z_2),\ldots,\varphi_n(z_1,z_2))=0 \in \bC[z_1,z_2]$
for any $i \in \{1,\ldots,r\}$.}
\end{align}
Since $f_i(\varphi_1(z_1,z_2),\ldots,\varphi_n(z_1,z_2)) \in \bC[z_1,z_2]$
is a homogeneous polynomial of degree $d l_i$ in $z_1$ and $z_2$,
it contains $d l_i + 1$ terms,
each of which is a homogeneous polynomial of degree $l_i$ in $(a_{kl})_{k,l}$.
With this in mind,
the \emph{Morrison-Plesser class}
is defined by
\begin{align}
 \Phi(Y;d)
  \coloneqq \prod_{i=1}^r (l_i x)^{l_i d}
  \in H^*(\Q(\bP^{n-1};d);\bZ),
\end{align}
so that $\Phi(Y;d) \cup v$ is the Poincar\'e dual of
$
 [\Q(Y;d)]^\virt \in H_*(\Q(\bP^{n-1};d);\bZ).
$
If we set
\begin{align}
 \la P(x) \ra_{Y,d}
  \coloneqq \int_{\Q(\bP^{n-1};d)} P(x) \cup \Phi(Y;d) \cup v
\end{align}
and
\begin{align}
 \la P(x) \ra_Y
  \coloneqq \sum_{d=0}^\infty q^d \la P(x) \ra_{Y,d}
\end{align}
for $P(x) \in \bC[x]$,
then we have
\begin{align}
 \la x^{n-r-1} \ra_Y
  &= \sum_{d=0}^\infty q^d \Res \frac{x^{n-r-1} \Phi(Y,d) v dx}{x^{n(d+1)}}
   \label{eq:Yuk_P_2} \\
  &= \sum_{d=0}^\infty q^d \Res \frac{x^{n-r-1} \prod_{i=1}^r (l_i x)^{l_i d + 1} dx}{x^{n(d+1)}} \\
  &= \sum_{d=0}^\infty q^d \prod_{i=1}^r (l_i)^{l_i d + 1} \\
  &= \frac{\prod_{i=1}^r l_i}{1 - q \prod_{i=1}^r (l_i)^{l_i}} . 
   \label{eq:Yuk_P_1}
\end{align}

\subsection{}

The gauged linear sigma model for $Y$
is obtained from the gauged linear sigma model for $\bP^{n-1}$
by adding $r$ fields
of $G=\bGm$-charge $- l_1, \ldots, - l_r$
and R-charge 2.
One has
$
 \Zvec_d(x) = 1
$
and
$
 \Zmat_d(x) = \lb x^{-d-1} \rb^n \cdot
  \prod_{i=1}^r \lb - l_i x \rb^{l_i d+1}
$
in this case,
so that
\pref{eq:correlator} gives
\begin{align}
 \sum_{d=0}^\infty
  e^{t' d}
  \Res \lb x^{-d-1} \rb^n \prod_{i=1}^r \lb -l_i x \rb^{l_i d + 1} x^{n-r-1}
  &= \sum_{d=0}^\infty e^{t' d} \prod_{i=1}^r \lb -l_i \rb^{l_i d + 1} \\
  &= \sum_{d=0}^\infty \lb -1 \rb^{\sum_{i=1}^r l_i d} e^{t' d}
   \prod_{i=1}^r \lb l_i \rb^{l_i d + 1} \\
  &= \sum_{d=0}^\infty \lb \lb -1 \rb^{n } e^{t'} \rb^d
   \prod_{i=1}^r \lb l_i \rb^{l_i d + 1},
\end{align}
which coincides with \pref{eq:Yuk_P_2}
under the identification
\begin{align} \label{eq:tt'4}
 q = (-1)^{n} e^{t'}.
\end{align}

\subsection{}

The mirror $\Yv$ of $Y$ is a compactification of
a complete intersection in $\bC^n$
defined by
\begin{align}
 \fv_1 &\coloneqq 1 - (a_1 \yv_1 + \cdots + a_{l_1} \yv_{l_1}), \\
 \fv_2 &\coloneqq 1 - (a_{l_1+1} \yv_{l_1+1} + \cdots + a_{l_1+l_2} \yv_{l_1 + l_2}), \\
 &\hspace{2mm} \vdots\\
 \fv_r &\coloneqq 1 - (a_{l_1+\cdots+l_{i-1}+1} \yv_{l_1+\cdots+l_{i-1}+1} + \cdots + a_n \yv_n), \\
 \fv_0 &\coloneqq \yv_1 \cdots \yv_n - 1.
\end{align}
The complex structure of $\Yv$ depends
not on the individual $a_i$ but only on $\alpha = a_1 \cdots a_n$.
The \emph{Yukawa (n-2)-point function} is defined by
\begin{align}
 \cY(\alpha)
  \coloneqq 
    \frac{(-1)^{(n-1)(n-2)/2}}{\lb 2 \pi \sqrt{-1} \rb^{n-1}}
   \int_{\Yv} \Omega \wedge \lb \alpha \frac{\partial}{\partial \alpha} \rb^{n-2} \Omega,
\end{align}
where
\begin{align}
 \Omega \coloneqq
  \Res \lb \frac{d \yv_1 \wedge \cdots \wedge d \yv_n}{\fv_0 \fv_1 \cdots \fv_r} \rb
\end{align}
is the holomorphic volume form on $\Yv$.
The computation in \cite[Proposition 5.1.2]{MR1328251} shows
\begin{align}
 \cY(\alpha)
  = \frac{\prod_{i=1}^r l_i}{1 - \alpha \prod_{i=1}^r (l_i)^{l_i}},
\end{align}
which coincides with \pref{eq:Yuk_P_1}
under the identification $q = \alpha$ of variables;
\begin{align} \label{eq:TRMS_P}
 \cY(\alpha) = \left. \la x^{n-r-1} \ra_Y \right|_{q = \alpha}.
\end{align}
A generalization of \pref{eq:TRMS_P} 
to toric complete intersections
is \emph{toric residue mirror symmetry}
conjectured in \cite{MR1988969,MR2019144}
and proved in \cite{MR2104791,MR2099774,MR2147350,MR2218757}.

\section{Concave bundles on projective spaces}
 \label{sc:concave}

\subsection{}

Let
$
l_1,l_2,\cdots,l_r
$
be positive integers and
%\begin{align}
% l_1 + l_2 + \cdots +l_r =n
%\end{align}
%and
\begin{align}
 Y \coloneqq \cSpec_{\bP^{n-1}} \lb \cSym^*  \cE^\vee \rb
\end{align}
be the total space of the vector bundle
associated with the locally free sheaf
\begin{align}
 \cE \coloneqq \cO_{\bP^{n-1}}(-l_1)
  \oplus \cdots \oplus \cO_{\bP^{n-1}}(-l_r)
\end{align}
on $\bP^{n-1}$.
%We write the Euler class of $\cE$ as
%\begin{align}
% v
%  \coloneqq \Eul(\cE)
%  = \prod_{i=1}^r (-l_i x).
%\end{align}
Since any holomorphic map
from $\bP^1$ to $Y$
of positive degree $d$
factors through the zero-section
$\bP^{n-1} \to Y$,
we define the quasimap space to $Y$ as
\begin{align}
 \Q(Y;d) \coloneqq \Q(\bP^{n-1};d).
\end{align}
%The quasimap space
%$\Q(Y;d)$ carries a natural perfect obstruction theory,
%which is different from that on $\Q(\bP^{n-1};d)$.

%{\red

\subsection{}

To equip $\Q(Y;d)$ with a natural obstruction theory,
we identify $\Q(Y;d)$
with an open substack of the mapping stack
$
 \Map(\bP^1, \cY)
$
to the quotient stack
\begin{align}
 \cY \coloneqq \ld \lb \bA^n \times \bA^r \rb / \bGm \rd
\end{align}
of $\bA^n \times \bA^r$
by the $\bGm$-action given by
\begin{align}
 \bGm \ni \alpha \colon (x_1, \ldots, x_n, z_1, \ldots, z_r)
  \mapsto (\alpha x_1, \ldots, \alpha x_n, \alpha^{-l_1} z_1, \ldots, \alpha^{-l_r} z_r).
\end{align}
A morphism
$
 \bP^1 \to \cY
$
consists of a line bundle $\cL$ on $\bP^1$
and sections
\begin{align}
 \lb (\varphi_i)_{i=1}^n, (\psi_j)_{j=1}^r \rb
  \in \lb (H^0(\cL))^n \times \prod_{j=1}^r H^0 \lb \cL^{\otimes (-l_j)} \rb \rb,
\end{align}
whose degree is defined as the degree of $\cL$.

\subsection{}

Recall from \cite[Definition 4.4]{MR1437495} that
an \emph{obstruction theory}
for a Deligne--Mumford stack $\cX$ is a morphism
$
 \phi \colon E \to L_\cX
$
from an object $E$
of the derived category of quasicoherent sheaves on $\cX$
satisfying
\begin{enumerate}
 \item
$h^i(E) \cong 0$ for $i > 0$, and
 \item
$h^i(E)$ is coherent for $i = 0, -1$
\end{enumerate}
to the cotangent complex $L_\cX$
such that
\begin{enumerate}
 \item
$h^0(\phi)$ is an isomorphism, and
 \item
$h^{-1}(\phi)$ is an epimorphism.
\end{enumerate}
It is said to be \emph{perfect}
if $E$ is  locally  isomorphic
to a two-term complex of locally free sheaves  of finite rank
\cite[Definition 5.1]{MR1437495}.

\subsection{}

%A \emph{relative perfect obstruction theory}
%of a morphism $\cX \to S$ of schemes
%is defined
%by replacing the cotangent complex
%$L_\cX$
%with the relative cotangent complex
%$L_{\cX/S}$.
A perfect obstruction theory produces
the \emph{virtual fundamental cycle}
$\ld \cX \rd^\virt$
in the Chow group
$A_{\vdim \cX}(\cX)$
of degree
\begin{align}
 \vdim \cX = \rank h^0(E) - \rank h^{-1}(E).
\end{align}
When $\cX$ is a smooth scheme,
then the cotangent complex $L_\cX$ is isomorphic
to the sheaf $\Omega_\cX$ of K\"ahler differentials,
and the virtual fundamental cycle
is the Euler class of $h^{-1}(E)$.

\subsection{}

The derived mapping stack
$\RMap(\cS,\cT)$
from a proper scheme $\cS$
to a derived Artin stack $\cT$
is a derived Artin stack
(see e.g.~\cite[Corollary 3.3]{MR3285853})
whose tangent complex is given by
\begin{align}
 T_{\RMap(\cS, \cT)} \cong \bR \pi_* \lb \bL \! \ev^* T_{\cT} \rb,
\end{align}
where
\begin{align}
 \pi \colon \RMap(\cS, \cT) \times \cS \to \RMap(\cS, \cT)
\end{align}
is the first projection and
\begin{align} \label{eq:ev1}
 \ev \colon \RMap(\cS, \cT) \times \cS \to \cT
\end{align}
is the evaluation morphism.
It is a derived thickening
of the mapping stack
$
 \Map(\cS, \cT),
$
and the pull-back
\begin{align}
 j^* \colon j^* L_{\RMap \lb \cS, \cT \rb}
  \to L_{\Map \lb \cS, \cT \rb}
\end{align}
by the structure morphism
\begin{align}
 j \colon \Map \lb \cS, \cT \rb
  \to \RMap \lb \cS, \cT \rb
\end{align}
gives an obstruction theory on $\Map \lb \cS, \cT \rb$.

\subsection{}

The restriction of the natural obstruction theory for
$
 \Map \lb \bP^1, \cY \rb
$
to the open substack
$
 \Q(Y;d)
%  \coloneqq \Q(\bP^{n-1};d)
$
gives an obstruction theory for $\Q(Y;d)$
with
$
 E = \left. j^* L_{\RMap(\bP^1,\cY)} \right|_{\Q(Y;d)}
$
and
$
 \phi =  j^*|_{\Q(Y;d)}.
$

\subsection{}

Since $\Pic \lb \bA^n \times \bA^r \rb$ is trivial,
the Picard group
$
 \Pic \cY \cong \Pic^{\bGm} \lb \bA^n \times \bA^r \rb
%  \coloneqq \ld \bA^n \times \bA^r / \bGm \rd,
$
%which can be identified with
%the $\bGm$-equivariant Picard group of $\bA^n \times \bA^r$,
can be identified
with the group of characters of $\bGm$,
which is non-canonically isomorphic to $\bZ$.
We fix an isomorphism in such a way that
$\bigoplus_{a=0}^\infty H^0(\cO_\cY(a))$
is the coordinate ring of $\bA^n$,
where $\cO_\cY(a)$ is the line bundle
associated with $a \in \bZ \cong \Pic \cY$.
Since $\cY$ is the quotient stack of $\bA^n \times \bA^r$
by the action of $\bGm$,
the tangent complex $T_\cY$ satisfies
\begin{align}
 \varpi^* T_\cY
  &\cong \Cone \lb \cO_{\bA^n \times \bA^r} \otimes \Lie(\bGm) \to T_{\bA^n \times \bA^r} \rb
\end{align}
where
$
 \varpi \colon \bA^n \times \bA^r \to \cY
$
is the quotient morphism.
This in turn implies that
\begin{align}
 T_\cY &\cong \Cone \lb \cO_\cY \to \cO_\cY(1)^{\oplus n} \oplus \bigoplus_{i=1}^r \cO_\cY(- l_i) \rb.
\end{align}
%in
%$
% D \Qcoh \cY \cong D \Qcoh^{\bGm} \lb \bA^n \times \bA^r \rb.
%$

\subsection{}

We write the restriction of the evaluation morphism
$
 \Map(\bP^1, \cY) \times \bP^1 \to \cY
$
to the open substack
$
 \Q(Y;d) \subset \Map(\bP^1, \cY)
$
as
\begin{align}
 \ev \colon \Q(Y;d) \times \bP^1 \to \cY
\end{align}
again by abuse of notation.
We have
\begin{align}
 \ev^* \cO_\cY(1) \cong \cO_{\Q(Y;d)}(1) \boxtimes \cO_{\bP^1}(d)
\end{align}
essentially by definition,
where $\cO_{\Q(Y;d)}(1)$ is the ample generator of
the Picard group of $\Q(Y;d) \cong \bP^{n(d+1)-1}$.
The dual of the natural obstruction theory is given by
\begin{align}
 \phi^\vee \colon T_{\Q(Y;d)} \to E^\vee \coloneqq \bR \pi_* \ev^* T_\cY.
\end{align}
Note that (the inverse of) the isomorphism $h^0(\phi^\vee)$ from $T_{\Q(Y;d)}$ to
\begin{align}
 h^0(E^\vee)
%  &\cong R^0 \pi_* \ev^* \Cone \lb \cO_\cY \to \cO_\cY(1)^{\oplus n} \rb \\
  &\cong R^0 \pi_* \Cone \lb \cO_{\Q(Y;d) \times \bP^1}
   \to \lb \cO_{\Q(Y;d)}(1) \boxtimes \cO_{\bP^1}(d) \rb^{\oplus n} \rb \\
  &\cong \Cone \lb \cO_{\Q(Y;d)} \to \cO_{\Q(Y;d)}(1)^{\oplus n(d+1)} \rb
\end{align}
gives the Euler sequence
\begin{align}
 0 \to \cO_{\Q(Y;d)} \to \cO_{\Q(Y;d)}(1)^{\oplus n(d+1)} \to T_{\cQ(Y;d)} \to 0
\end{align}
on $\Q(Y;d) \cong \bP^{n(d+1)-1}$.
One has
\begin{align}
 h^1(E^\vee)
%  &\cong R^1 \pi_* \ev^* \lb \bigoplus_{i=1}^r \cO_\cY(-l_i) \rb \\
  &\cong R^1 \pi_* \lb \bigoplus_{i=1}^r \cO_{\Q(Y;d)}(-l_i) \boxtimes \cO_{\bP^1}(-l_id) \rb \\
  &\cong \bigoplus_{i=1}^r \cO_{\Q(Y;d)}(-l_i) \otimes H^1 \lb \cO_{\bP^1}(-l_id) \rb \\
  &\cong \bigoplus_{i=1}^r \cO_{\Q(Y;d)}(-l_i)^{\oplus (l_i d - 1)}
\end{align}
and $h^i(E^\vee) \cong 0$ for $i \ne 0, 1$,
so that this obstruction theory is perfect.
By 
\cite[Proposition 5.6]{MR1437495},
the resulting virtual fundamental class is given by
\begin{align}
 \ld \Q(Y;d) \rd^\virt
  = \ld \Q(Y;d) \rd \cap \Eul \lb h^1(E^\vee) \rb
  = \ld \Q(Y;d) \rd \cap \prod_{i=1}^r (-l_i x)^{l_i d -1}.
\end{align}

\subsection{}

When the degree is zero,
the quasimap space
$\Q(Y; 0)$
is naturally isomorphic to $Y$
equipped with the trivial perfect obstruction theory,
so that
\begin{align}
 \ld \Q(Y;0) \rd^\virt
  = [Y].
\end{align}

\subsection{}

For any $P(x) \in \bC[x]$,
we define
\begin{align}
 \la P(x) \ra_{Y,d}
  \coloneqq \int_{[\Q(Y;d)]^\virt} P(x)
\end{align}
and
\begin{align}
 \la P(x) \ra_{Y}
  \coloneqq \sum_{d=0}^\infty q^d \la P(x) \ra_{Y,d}. 
\end{align}
It follows that
\begin{align}
 \la P(x) \ra_Y
  &= \sum_{d=0}^\infty q^k
   \int_{\bP^{n(d+1)-1}} P(x) \prod_{i=1}^r (-l_i x)^{l_i d -1} \\
  &= \sum_{d=0}^\infty q^k
  \Res \frac{P(x) \prod_{i=1}^r (-l_i x)^{l_i d}}{x^{n(d+1)} \prod_{i=1}^r (-l_i x) }.
   \label{eq:concave_cor1}
\end{align}

\subsection{}

The gauged linear sigma model for $Y$
is obtained from the gauged linear sigma model for $\bP^{n-1}$
by adding $r$ fields
of $G=\bGm$-charge $- l_1, \ldots, - l_r$
and R-charge 0.
One has
$
 \Zvec_d(x) = 1
$
and
$
 \Zmat_d(x) = \lb x^{-d-1} \rb^n \cdot
  \prod_{i=1}^r \lb - l_i x \rb^{l_i d-1}
$
in this case,
so that
\pref{eq:correlator} gives
\begin{align}
 \cor{P(x)}
 &= \sum_{d=0}^\infty
  e^{t' d}
  \Res \lb x^{-d-1} \rb^n \prod_{i=1}^r \lb -l_i x \rb^{l_i d - 1} P(x),
\end{align}
which coincides with \pref{eq:concave_cor1}
under the identification
\begin{align} \label{eq:tt'5}
 q = e^{t'}.
\end{align}

\subsection{}

If $(l_1, \ldots, l_r)$ satisfies the Calabi--Yau condition
\begin{align}
 l_1 + \cdots + l_r = n,
\end{align}
then \pref{eq:concave_cor1} gives
\begin{align} \label{eq:concave1}
 \la x^k \ra_Y
  =
\begin{cases}
 \dfrac{1}{\lb \prod_{i=1}^r (- l_i) \rb \lb 1 - q \prod_{i=1}^r (- l_i)^{l_i} \rb} & k = n + r, \\[4mm]
 0 & \text{otherwise},
\end{cases}
\end{align}
which matches the Yukawa coupling of the mirror
(see e.g. \cite[Example 6.15]{MR2821394}).

\section{Classical mirror symmetry
for toric hypersurfaces}
 \label{sc:CMSTH}

\subsection{}

Let
$
 \bsN \coloneqq \bZ^n
$
be a free abelian group of rank $n$ and
$
 \bsM
  \coloneqq \bsNv
  \coloneqq \Hom(\bsN, \bZ)
$
be the dual group.
Let further $(\Delta, \Deltav)$ be a polar dual pair
of reflexive polytopes in $\bsM$ and $\bsN$.

\subsection{}

Recall that the \emph{fan polytope}
of a fan is defined as the convex hull
of primitive generators of one-dimensional cones.
Let $(\Sigma, \Sigmav)$ be a pair of smooth projective fans
whose fan polytopes are $\Deltav$ and $\Delta$.
The associated toric varieties will be denoted by
$
 X \coloneqq X_\Sigma
$
and
$
 \Xv \coloneqq X_{\Sigmav}.
$

\subsection{}

The set of primitive generators
of one-dimensional cones of the fan $\Sigma$
will be denoted by
\begin{align}
 B \coloneqq \lc \bb_1, \ldots, \bb_m \rc
  \subset \bsN.
\end{align}
Assume that $B$ generates $\bsN$.
One has the \emph{fan sequence}
\begin{align}
 0 \to \bsL \to \bZ^m \xto{\bb} \bsN \to 0
\end{align}
and the \emph{divisor sequence}
\begin{align}
 0 \to \bsM \xto{\bb^\vee} \bZ^m  \xto{} \bsLv \to 0,
\end{align}
where $\bb$ sends the $i$th coordinate vector $e_i \in \bZ^m$
to $\bsb_i$. Recall that
\begin{align}\label{EffCurve}
 \bsLv \cong \Pic(X) \cong H^2(X; \bZ), \ \ \mathrm{Eff}(X) \subset \bsL \subset \mathbb{Z} ^m,
\end{align}
where  $\mathrm{Eff}(X)$  denotes the semigroup of the effective curves (see \cite[\S 3]{MR1988969}).
We write the group ring of $\bsM$ as $\bC[\bsM]$
and define
$
 \bT
  \coloneqq \bsN_{\bGm}
  \coloneqq \Spec \bC[\bsM].
$
We also set
$
 \bTv
  \coloneqq \Spec \bC[\bsN]
$,
$
 \bLv \coloneqq \Spec \bC[\bsL],
$
and
$
 \bL \coloneqq \Spec \bC [ \bsLv ].
$
The fan sequence induces the exact sequences
\begin{align} \label{eq:fan2}
 1 \to \bL \xto{\bchi} (\bGm)^m \to \bT \to 1
\end{align}
and
\begin{align}
 1 \to \bTv \to (\bGm)^m \to \bLv \to 1
\end{align}
of algebraic tori.
We write the $i$-th components of the map
$
 \bchi \colon \bL \to (\bGm)^m
$
in \pref{eq:fan2}
as $\chi_i$, and
the affine line $\bA^1$
equipped with the action of $\bL$ through $\chi_i$
as $\bA_i$.
Then one has
%The toric variety $X$ is isomorphic
%to the GIT quotient
\begin{align} \label{eq:toric_GIT}
 X \cong \lb \prod_{i=1}^m \bA_i \rb \GIT_\theta \, \bL
\end{align}
for a suitable choice of a character
$
 \theta \in \bsLv \cong \Hom(\bL, \bGm).
$
The right-hand side of \eqref{eq:toric_GIT} denotes the GIT quotient with respect to the linearization determined by $\theta$. 

\subsection{}

We define a graded ring
$
 S_\Delta \coloneqq \bigoplus_{k=0}^\infty S_\Delta^k
$
by
\begin{align}
 S_\Delta^k
  \coloneqq \bigoplus_{\bm \in \bsM \cap (k \Delta)} \bC \cdot y_0^k \by^\bm,
\end{align}
which is a subalgebra
of the semigroup ring
\begin{align}
 \bC[\bN \times \bsM]
  = \bC[y_0, \by^{\pm 1}]
  \coloneqq \bC[y_0, y_1^{\pm 1}, \ldots, y_n^{\pm 1}]
\end{align}
of $\bN \times \bsM$.
It is the anti-canonical ring of $X$,
so that one has $X \cong \Proj S_\Delta$
if and only if $X$ is Fano.
The ring $S_\Delta$ is Cohen-Macaulay
with the dualizing module
$
 I_\Delta \coloneqq \bigoplus_{k=0}^\infty I_\Delta^k
$
given by
\begin{align}
 I_\Delta^k
  \coloneqq \bigoplus_{\bm \in \bsM \cap \Int(k \Delta)} \bC \cdot y_0^k \by^\bm,
\end{align}
where $\Int(k \Delta)$ is the interior of $k \Delta$.

\subsection{}

For $\balpha = (\alpha_1, \ldots, \alpha_m) \in (\bGm)^m$ (this $(\bGm)^m$ can be naturally considered 
as the dual torus of the big torus of $X_\Sigma$),
we define an element of the group ring $\bC[\bsN]$
%of $\bsN$
by
\begin{align}
 \Wv_\balpha(\byv)
  \coloneqq \sum_{i=1}^m \alpha_i \byv^{\bsb_i}
  \in \bC[\bsN].
\end{align}
%\begin{align}
% \bC^B \coloneqq
%  \lc W(\by) \in \bC[\bsN] \relmid
%   W(\by) = \sum_{\bn \in B} \alpha_\bn \by^\bn \rc
%\end{align}
%be the set of Laurent polynomials
%supported on $B$.
An element
$
 \fv \in \bC[\bsN]
$
is said to be \emph{$\Deltav$-regular}
if
\begin{align} \label{eq:Fv}
 \Fv \coloneqq (\Fv_0, \Fv_1, \ldots, \Fv_n)
  \coloneqq \lb
   \yv_0 \fv, \yv_0 \yv_1 \partial_{\yv_1} \fv, \ldots,
   \yv_0 \yv_n \partial_{\yv_n} \fv
  \rb
\end{align}
is a regular sequence in $S_\Deltav$.
We write
\begin{align}
 \lb (\bGm)^m \rb^\reg \coloneqq
  \lc \balpha \in (\bGm)^m \relmid
%   f \coloneqq
   \fv_\bsalpha \coloneqq 1 - \Wv_\balpha(\byv) \text{ is $\Deltav$-regular} \rc.
\end{align}
%The set of $W \in \bC^B$
%such that $f$ is $\Delta$-regular will be denoted by
%$(\bC^B)^\reg$.

\subsection{}

Let
$
 \widetilde{\varphiv} : \widetilde{\fYv} \to \lb (\bGm)^m \rb^\reg
$
be the second projection from
\begin{align}
 \widetilde{\fYv} = \lc (\byv, \balpha) \in \bTv \times \lb (\bGm)^m \rb^\reg
   \mid \Wv_\balpha(\byv) = 1 \rc.
\end{align}
Assume that $X$ is Fano.
Any fiber
$
 \Yv_\balpha \coloneqq \widetilde{\varphiv}^{-1} (\balpha)
$
is an uncompactified mirror
of a general anti-canonical hypersurface $Y \subset X$.
The closure of $\Yv_\balpha$ in $\Xv$
is a smooth anti-canonical Calabi--Yau hypersurface,
which is the compact mirror of $Y$.
The quotient of the family
$
 \widetilde{\varphiv} \colon \widetilde{\fYv} \to \lb (\bGm)^m \rb^\reg
$
by the free $\bTv$-action
\begin{align}
 \bTv \ni \byv \colon \lb \byv', (\alpha_1, \ldots, \alpha_m) \rb
  \mapsto \lb \byv^{-1} \byv',
   \lb \byv^{\bsb_1} \alpha_1, \ldots, \byv^{\bsb_m} \alpha_m \rb \rb
\end{align}
will be denoted by
$
 \varphiv \colon \fYv \to \bLv^\reg,
$
where
$
 \fYv \coloneqq \widetilde{\fYv} / \bTv
$
and
$
 \bLv^\reg \coloneqq \lb (\bGm)^m \rb^\reg / \bTv.
$

\subsection{}

Choose an integral basis $\bp_1, \ldots, \bp_r$ of $\bsLv \cong \Pic X$
such that each $\bp_i$ is nef.
This gives the corresponding coordinate
$\bq = (q_1, \ldots, q_r)$ on $\bLv$.
Let $\Uv' \subset \bLv^\reg$ be a neighborhood
of $q_1 = \cdots = q_r = 0$,
and $\Uv$ be the universal cover of $\Uv'$.

\subsection{}

We write the image of the Poincar\'e residue as
\begin{align}
 H^{n-1}_\res(\Yv_\alpha)
  \coloneqq \Image \lb
   \Res \colon H^0 \lb \Xv, \Omega_\Xv^n \lb * \Yv_\alpha \rb \rb
    \to H^{n-1} \lb \Yv_\alpha \rb
  \rb.
\end{align}
Let $\HB$ be the pull-back to $\Uv$
of the local system
$
 \gr_{n-1}^W \! R^{n-1} \varphiv_! \, \bC_\fYv
$
on $\Uv'$,
and $\HBres$ be the sub-system
with stalks $H^{n-1}_\res(\Yv_\alpha)$.
Here $\gr_{n-1}^W$ is the weight $n-1$ piece
of Deligne's mixed Hodge structure.
The \emph{residual B-model VHS}
$
 (\cHB, \nablaB, \FB^\bullet, \QB)
$
on $\Uv$ consists of the locally free sheaf
$\cHB \coloneqq \HBres \otimes_\bC \cO_\Uv$,
the Gauss--Manin connection $\nabla_B$,
the Hodge filtration $\scrF_B^\bullet$, and
the polarization
$
 Q_B \colon \cHB \otimes_{\cO_{\Uv}} \cHB \to \cO_{\Uv}
$
given by
\begin{align}
 \QB(\omega_1, \omega_2)
  \coloneqq (-1)^{(n-1)(n-2)/2} \int_{\Yv_{\alpha}} \omega_1 \cup \omega_2.
\end{align}

\subsection{}

On the A-model side, let
\begin{align}
 H^\bullet_\amb(Y; \bC)
  \coloneqq \Image(\iota^* \colon H^\bullet(X; \bC) \to H^\bullet(Y; \bC))
\end{align}
be the subspace of $H^\bullet(Y; \bC)$
coming from the cohomology classes of the ambient toric variety,
and set
\begin{align}
 U \coloneqq \lc \btau = \bsbeta + \sqrt{-1} \bsomega \in H^2_\amb(Y; \bC)
  \relmid
   \pair{\bsomega}{\bd} \gg 0 \text{ for any non-zero }
   \bd \in \Eff(Y) \rc .
\end{align}
%where $\Eff(Y)$ is the semigroup of effective curves.
This open subset $U$ is considered
as a neighborhood of the large radius limit point.
Let $(\tau_i)_{i=1}^r$ be the coordinate on $H^2_\amb(Y; \bC)$
dual to the basis $\{ \bp_i \}_{i=1}^r$
so that
$
 \btau = \sum_{i=1}^r \tau_i \bp_i.
$

\subsection{}

The \emph{ambient A-model VHS}
$(\cHA, \nablaA, \FA^\bullet, \QA)$
consists
(\cite[Definition 6.2]{MR3112512},
cf. also \cite[Section 8.5]{MR1677117})
of the locally free sheaf
$
 \cHA = H_\amb^\bullet(Y;\bC) \otimes_\bC \cO_U,
$
the connection
\begin{align}
 \nablaA = d + \sum_{i=1}^r (\bp_i \circ_\btau) \, d \tau_i
  \colon \cHA \to \cHA \otimes \Omega_{U}^1,
\end{align}
the Hodge filtration
\begin{align}
 \FA^p \coloneqq H_\amb^{\le 2(n-1-p)}(Y;\bC) \otimes_\bC \cO_U,
\end{align}
and the pairing
\begin{align}
 Q_A \colon \cHA \otimes_{\cO_U} \cHA \to \cO_U, \qquad
(\alpha, \beta) \mapsto (2\pi\sqrt{-1})^{n-1}  \int _Y (-1)^{\deg \alpha/2} \alpha \cup \beta,
\end{align}
%$$
%\begin{array}{cccc}
% Q_A : & \HA \otimes \HA & \to & \cO_U \\
%  & \vin & & \vin \\
%  & \alpha \otimes \beta & \mapsto &
%  \dfrac{1}{4 \pi^2}(\alpha, \beta)_{\mathrm{Mukai}}
%\end{array}
%$$
which is $(-1)^{n-1}$-symmetric and $\nablaA$-flat.
%$$
%\begin{array}{cccc}
% Q_A : & \HA \otimes \HA & \to & \cO_{U'} \\
%  & \vin & & \vin \\
%  & (\alpha, \beta) & \mapsto &
% (2 \pi \sqrt{-1})^{n-1} (-1)^{\deg \alpha / 2}
%  (\alpha, \beta).
%\end{array}
%$$
%$
% Q_A : \HA \otimes \HA \to \cO_{U'}
%$
%$
% Q_A(\alpha, \beta) = (2 \pi \sqrt{-1})^{n-1}
%  ((-1)^{\deg \alpha / 2} \alpha, \beta)
%$

\subsection{}

Let $\bu_i \in H^2_\amb(Y; \bZ)$ be the first Chern class
of the line bundle on $Y$
corresponding to the one-dimensional cone
$\bR_{\ge 0} \cdot \bsb_i \in \Sigma$ and
$\bsv = \bu_1 + \cdots + \bu_m$
be the restriction of the anti-canonical class of $X$.
Denote $\bt := \sum_{i=1}^r t_i\bp_i $.
\emph{Givental's $I$-function} is defined as the series
\begin{align} \label{eq:Givental_I}
 I_Y(\bt;\z) = e^{\bt / \z}
  \sum_{\bd \in \Eff(X)} e^{\inner{\bd}{\bt}} \,
  \frac{
   \prod_{k=-\infty}^{\la \bd, \bsv \ra} (\bsv + k \z)
   \prod_{j=1}^m \prod_{k=-\infty}^0
    (\bu_j + k \z)}
  {\prod_{k=-\infty}^0
    (\bsv + k \z)
   \prod_{j=1}^m \prod_{k=-\infty}^{\la \bd, \bu_j \ra}
    (\bu_j + k \z)},
\end{align}
which is a multi-valued map from $\Uv'$
(or a single-valued map from $\Uv$)
to the classical cohomology group $H^\bullet_\amb(Y; \bC[\z^{-1}])$.
The \emph{$J$-function} is defined by
\begin{align}
 J_Y(\btau; \z) = L_Y(\btau, \z)^{-1}(1),
\end{align}
where $L_Y(\btau,\z)$ is the fundamental solution
of the quantum differential equation
defined explicitly by using the Gromov--Witten invariants
as in \cite[Equation (2.3)]{MR3112512}
with $\bf c$ set to $1$.
If we write
\begin{align}
 I_Y(\bt;\z) = F(\bt) \bsone + \frac{G(\bt)}{\z} + O(\z^{-2}),
\end{align}
then Givental's mirror theorem
\cite{MR1653024}
states that
\begin{align} \label{eq:Givental_MS}
 I_Y(\bt; \z)
  = F(\bt) \cdot J_Y(\bvarsigma(\bt); \z),
\end{align}
where
%$\Euler(\omega_X^\vee) \in H^2(X; \bZ)$
%is the Euler class of the anti-canonical bundle of $X$,
%and
the \emph{mirror map}
$
 \bvarsigma : \Uv \to H^2_\amb(Y; \bC)
$
is defined by
\begin{align}
 \bvarsigma(\bt) = \iota^* \lb \frac{G(\bt)}{F(\bt)} \rb.
\end{align}
The relation between $\btau = \bvarsigma(\bt)$ and
$\bsigma = \bsbeta + \sqrt{-1} \bsomega$ is given by
$\btau = 2 \pi \sqrt{-1} \bsigma$,
so that $\Im(\bsigma) \gg 0$ corresponds to $\exp(\btau) \sim 0$.
%It is straightforward to see that
The functions $F(\bt)$ and $G(\bt)$ satisfy
%the Gelfand--Kapranov--Zelevinsky hypergeometric differential equations,
the Picard--Fuchs equations,
and give periods for the B-model VHS
$
 (\cHB, \nabla^B, \scrF_B^\bullet, Q_B).
$

\subsection{}

%Givental's mirror theorem
\pref{eq:Givental_MS}
implies the existence of an isomorphism
\begin{align} \label{eq:CMS}
 \Mir_\scY : \varsigma^*
  (\cHA, \nablaA, \FA^\bullet, Q_A)
%    /\iota^*H^2(\scX; \bZ)
  \simto (\cHB, \nablaB, \FB^\bullet, Q_B)
\end{align}
of variations of polarized Hodge structures,
which sends $F(\bt) \bsone$ on the left-hand side to
\begin{align}
  \Omega \coloneqq \Res \lb \frac{1}{\fv_\bsalpha}
  \frac{d \yv_1}{\yv_1} \wedge
  \cdots \wedge \frac{d \yv_n}{\yv_n} \rb
\end{align}
on the right-hand side.
A stronger statement,
which gives an isomorphism of
the $\Gammahat$-integral structure on the A-side
and the natural integral structure on the B-side,
is proved in \cite[Theorem 6.9]{MR3112512}.

\section{Quasimap correlation functions for anti-canonical hypersurfaces in toric varieties}
 \label{sc:QSTV}

\subsection{}

%We use the same notations as
%%Let $X$ be the toric variety
%%associated with the fan $\Sigma$ appearing
%in \pref{sc:CMSTH}.
For $\bd \in \Eff(X)$ and $i \in \{ 1, \ldots, m \}$,
we set
\begin{align}
 k_i \coloneqq
\begin{cases}
 \pair{\bu_i}{\bd} & \pair{\bu_i}{\bd} \ge 0, \\
 -1 & \pair{\bu_i}{\bd} < 0.
\end{cases}
\end{align}
and define the \emph{quasimap space} of degree $\bd$ by
\begin{align}
 \Q(X; \bd) \coloneqq
  \lb \prod_{i=1}^m \bA_i^{k_i+1} \rb \GIT_\theta \, \bL
%   = \lb \prod_{i=1}^m \bA_i^{b_i+1} \rb^\semistable / \bL
\end{align}
with \pref{eq:toric_GIT} in mind.
An argument parallel to that in Section \ref{QmaptoP} shows that
$\Q(X;\bd)$ is a compactification of the space of holomorphic maps $\bP^1 \to X$ of degree $\bd$.
%{\red
Later in Section \ref{Quasimap},
we will introduce the moduli spaces
$\Q(W \GIT G; \bd)$
of degree $\bd$ quasimaps from $\bP^1$
to more general GIT quotients $W \GIT G$,
and $\Q(X;\bd)$ here is the special case
for $W= \prod_{i=1}^m \bA_i^{k_i+1}$ and $G=\bL$. %}
The first Chern class of the line bundle on $\Q(X;\bd)$
associated with the character $\chi_i$ of $\bL$
will also be denoted by $\bu_i$ by abuse of notation.
The \emph{Morrison--Plesser class}
is defined by
\begin{align} 
 \Phi_\bd \coloneqq
  (\bu_1+\cdots+\bu_m)^{\pair{\bu_1+\cdots+\bu_m}{\bd}}
  \prod_{\pair{\bu_i}{\bd} < 0} \bu_i^{-\pair{\bu_i}{\bd}-1}.
\end{align}
Here, the latter part $\prod_{\pair{\bu_i}{\bd} < 0} \bu_i^{-\pair{\bu_i}{\bd}-1}$ is the Euler class
of $h^1(E^\vee)$ where $E$ is the canonical obstruction theory for $\Q(X;\bd)$ defined in Section 10.7.
%{\red of the $(-1)$-th cohomology of the natural obstruction theory for}
%of the vector bundle formed by the obstruction spaces of
%$\Q(X;\bd)$.
The first part $(\bu_1+\cdots+\bu_m)^{\pair{\bu_1+\cdots+\bu_m}{\bd}}$ is the Euler class of the vector bundle formed by the obstruction spaces to being a quasimap to an anti-canonical hypersurface $Y \subset X$ for each element in $\Q(X;\bd)$.
For a polynomial
$P(\alpha_1, \ldots, \alpha_m) \in \bC[\alpha_1, \ldots, \alpha_m]$,
%in $m$ variables,
we set
\begin{align}
 \la P(\bu_1, \ldots, \bu_m) \ra_{X,Y,\bd}
  \coloneqq \int_{\Q(X;\bd)} P(\bu_1, \ldots, \bu_m) \Phi_\bd
\end{align}
and
%define a formal power series in $\bsalpha$ by
\begin{align} \label{eq:A-Yukawa0}
 \la P(\bu_1, \ldots, \bu_m) \ra_{X,Y}
  \coloneqq \sum_{\bd \in \Eff(X)} \balpha^\bd
   \la P(\bu_1, \ldots, \bu_m) \ra_{X,Y,\bd}
  \in \bZ \db[ \alpha ^{\bd} : \bd \in \Eff(X) \db],
\end{align}
where the completion is taken with respect to the ideal generated by $\Eff (X) \setminus \{0\}$.
Here $\balpha ^{\bd}$ is defined by \eqref{EffCurve}.

\section{Toric residue mirror symmetry}
 \label{sc:TR}

%In \pref{sc:TR},
%we recall the definition of toric residues
%introduced in \cite{MR1396624}.
%Our exposition follows \cite{MR1988969} closely.

\subsection{}

Let
$
 \Gv = (\Gv_0, \ldots, \Gv_n)
$
be a regular sequence in $S_\Deltav$.
If we set
$
 I_\Gv \coloneqq I_\Deltav/(\Gv_0, \ldots, \Gv_n) I_\Deltav,
$
then the graded piece
$
 I^{n+1}_\Gv
$
is one-dimensional and spanned by
$
 J_\Gv \coloneqq \det \lb \yv_i \partial_{\yv_i} \Gv_j \rb_{i,j=0}^n.
$
The \emph{toric residue}
\cite{MR1396624}
is the map
$
 \Res_\Gv \colon I_\Deltav^{n+1} \to \bC
$
sending
$
 (\Gv_0, \ldots, \Gv_n) I_\Deltav
$
to zero and
$J_\Gv$
to the normalized volume
$\vol(\Deltav)$,
i.e.,
$n!$ times the Euclidean volume of $\Deltav$.
For $\balpha \in \bLv^\reg$,
we define $\Fv_\balpha$
as in \pref{eq:Fv}
and write
$
 \Res_{\fv_\balpha} \coloneqq \Res_{\Fv_\bsalpha}.
$
\pref{th:TRMS} below is introduced
%under the name \emph{toric residue mirror conjecture}
in \cite[Conjecture 4.6]{MR1988969}
and proved
in \cite{MR2104791,MR2099774}.

\begin{theorem}%[{\cite[Conjecture 4.6]{MR1988969}}]
 \label{th:TRMS}
For any homogeneous polynomial
$P(\alpha_1, \ldots, \alpha_m) \in \bC[\alpha_1, \ldots, \alpha_m]$
of degree $n$,
the generating function
\pref{eq:A-Yukawa0} gives the Laurent expansion
of the toric residue
\begin{align}
 \la P(\bu_1, \ldots, \bu_m) \ra_{X,Y}
  = (-1)^n \Res_{\fv_\bsalpha} \lb \yv_0^{n+1}
      P(\alpha_1 \byv^{\bsb_1}, \ldots, \alpha_m \byv^{\bsb_m}) \rb
\end{align}
around the large radius limit point
associated with the fan $\Sigma$.
\end{theorem}

\cite[Conjecture 4.6]{MR1988969} is generalized to toric complete intersections
in \cite[Conjecture 4.6]{MR2019144}
and proved in \cite{MR2147350,MR2218757}.
%\cite[Theorem 6.1]{MR2147350}.

\subsection{}

The family $\varphi \colon \cYv \to \bLv^\reg$
of Calabi--Yau manifolds comes
with the holomorphic volume form
\begin{align}
 \Omega \coloneqq \Res \lb \frac{1}{\fv_\bsalpha}
  \frac{d \yv_1}{\yv_1} \wedge
  \cdots \wedge \frac{d \yv_n}{\yv_n} \rb
  \in H^0(\cHB). %H^0 \lb R^{d-1} \varphi_* \bC \rb.
\end{align}
For a homogeneous polynomial
$
 Q(\alpha_1, \ldots, \alpha_m) \in \bQ[\alpha_1, \ldots, \alpha_m]
$
of degree $n-1$,
the $Q$-Yukawa $(n-1)$-point function is defined
in \cite[Definition 9.1]{MR1988969}
by
\begin{align} \label{eq:B-Yukawa}
% Y_Q(\alpha_1, \ldots, \alpha_m) \coloneqq
 Y_Q(\balpha) \coloneqq
 (-1)^{(n-1)(n-2)/2} \frac{1}{\lb 2 \pi \sqrt{-1} \rb^{n-1}}
 \int_{\Yv_\balpha} \Omega \wedge Q \lb \alpha_1 \frac{\partial}{\partial \alpha_1},
  \ldots, \alpha_m \frac{\partial}{\partial \alpha_m} \rb \Omega,
\end{align}
where
the differential operators
$\alpha_1 \partial/\partial \alpha_1, \ldots, \alpha_m \partial/\partial \alpha_m$
act on
$\cHB$
%$R^{d-1} \varphi_* \bC \otimes_\bC \cO_{\bLv^\reg}$
by the Gauss--Manin connection.
%$\nablaB$.

\subsection{}

For $Q(\alpha_1, \ldots, \alpha_m) \in \bQ[\alpha_1, \ldots, \alpha_m]$,
we set
\begin{align}
 P(\alpha_1, \ldots, \alpha_m)
  \coloneqq (\alpha_1+\cdots+\alpha_m) Q(\alpha_1, \ldots, \alpha_m)
  \in \bQ[\alpha_1, \ldots, \alpha_m].
\end{align}
By \cite[Theorem 9.7]{MR1988969},
which is attributed to \cite{MR1733735},
one has an equality
\begin{align} \label{eq:TR=Yuk}
 Y_Q(\bsalpha)
  = (-1)^n \Res_{\fv_\balpha} \lb \yv_0^n
  P(\alpha_1 \byv^{\bsb_1}, \ldots, \alpha_m \byv^{\bsb_m}) \rb
%  (\alpha_1+\cdots+\alpha_m) \qtilde(z) \cdot \prod_{i=1}^r z_i \rb
\end{align}
of the Yukawa $(n-1)$-point function
and the toric residue.

%\section{Toric residue mirror symmetry and classical mirror symmetry}

\begin{comment}

\subsection{}

In Givental's notation,
one has
\begin{align} \label{eq:Givental_factorization0}
 \Phi(\bz,\bq;\z)
  \coloneqq \sum_{\bd \in \Lambda} \bq^\bd \Phi_\bd, \quad
 \Phi_\bd \coloneqq \int_{[Y_\bd]} e^{\inner{\bp}{\bz}}
\end{align}
and
\begin{align}
 \int_Y
  \Psi(\bq \exp(\z \bz); \z)
  e^{\bp \bz}
  \Psi(\bq, -\z; \z)
 = \Phi(\bz, \bq; \z).
\end{align}
If we set
$
 \bq = \exp(\btau)
$
and
$
 \bz = (\bt-\btau)/\z,
$
then one has
\begin{align}
 \Psi(\bq \exp(\z \bz);\z)
  &= \Psi(\exp(\bt);\z)
  = \exp(-\inner{\bp}{\bt}/\z) I(\exp(\bt);\z)
\end{align}
and
\begin{align}
 \Psi(\bq;-\z)
  &= \exp(\inner{\bp}{\btau}/\z) I(\exp(\btau);-\z),
\end{align}
so that \pref{eq:Givental_factorization0} gives
\begin{align}
 \Psi((\bt-\btau)/\z,\bq;\z)
  = \int_Y I(\exp(\bt);\z) I(\exp(\btau);-\z).
\end{align}
\end{comment}

\subsection{}

%Let
%\begin{align}
% \Psi_Y(\bt;\z)
%  \coloneqq
% \sum_{\bd \in \Lambda} e^{\inner{\bd}{\bt}}
%\frac{
%%\prod_{a=1}^l
%\prod_{m=-\infty}^{L_a(\bd)}
% \lb \bv + m \z \rb
%\prod_{j=1}^N \prod_{m=-\infty}^{0}
% \lb \bu_j + m \z \rb
%}
%{
%%\prod_{a=1}^l
%\prod_{m=-\infty}^{0}
% \lb \bv + m \z \rb
%\prod_{j=1}^N \prod_{m=-\infty}^{D_j(\bd)}
% \lb \bu_j + m \z \rb
%},
%\end{align}
%be the formal series
%defined in \cite[Section 6]{MR1653024},
%which is related to
%%the $I$-function in
%\pref{eq:Givental_I}
%by
%\begin{align}
% I_Y(\bt;\z) = e^{\inner{\bp}{\bt}/\z} \Psi(\bt;\z).
%\end{align}
%\begin{align}
% I_\cV(t_0, \bt;\z) =
%e^{t_0 + \inner{\bp}{\bt}}
%\sum_{\bd \in \Lambda}
%e^{\inner{\bt}{\bd}}
%\frac{
%\prod_{a=1}^l \prod_{m=-\infty}^{L_a(\bd)}
% \lb v_a + m \z \rb
%\prod_{j=1}^N \prod_{m=-\infty}^{0}
% \lb u_j + m \z \rb
%}
%{
%\prod_{a=1}^l \prod_{m=-\infty}^{0}
% \lb v_a + m \z \rb
%\prod_{j=1}^N \prod_{m=-\infty}^{D_j(\bd)}
% \lb u_j + m \z \rb
%}
%\end{align}

Assume that
the unstable locus of the $\bL$-action on $\bA^m$
with respect to $\theta$
has codimension strictly greater than $1$.
Then one has
$
 H^2 (X_\Sigma) = \Pic (X_\Sigma ) = \Pic^\bL (\bA^m)
$
so that
the class $\bp _i$ corresponds to a one-dimensional
representation $\bC_{\bp _i}$ of $\bL$.
By abuse of notation,
we let $\bp _i$ denote the $\bGm$-equivariant Euler class
of the pull-back of the line bundle
$[\bA^m \times \bC_{\bp _i} / \bGm ]$
by the evaluation map
$
 \ev_0 \colon X_{\bd} \to [\bA^m  / \bGm ]
$
at $0 \in \bP^1$. Denote $\bv := \sum_{i=1}^m \bu _i $.

If we set
\begin{align} \label{eq:Givental_Phi}
 \Phi(\bt,\btau;\z)
  \coloneqq
 \sum_{\bd \in \Eff(X)} e^{\inner{\btau}{\bd}}
 \int_{\Q(X;\bd)}^{{ \bGm}} e^{(\bt-\btau)/\z} \Phi _{\bd} \bv,
\end{align}
then for any polynomial
$R(t_1, \ldots, t_r) \in \bQ[t_1, \ldots, t_r]$,
one has
\begin{align}
 \left. R \lb \z \frac{\partial}{\partial t_1}, \ldots, \z \frac{\partial}{\partial t_r} \rb
  \Phi(\bt,\btau;\z) \right|_{\btau = \bt}
  &= 
 \sum_{\bd \in \Eff(X)} e^{\inner{\bt}{\bd}}
  \int_{\Q(X;\bd)}^{{ \bGm}} R (\bp_1, \ldots, \bp_r) \Phi_{\bd} \bv.
\end{align}
In addition,
one has
\begin{align} \label{eq:Givental_factorization}
 \Phi(\bt,\btau;\z)
  =
 \int_Y
  I(\bt;-\z) \cup I(\btau;\z),
\end{align}
by \cite[Proposition 6.2]{MR1653024}. This is the toric hypersurface version of \eqref{eqn:decom Phi}.
Note that  $I(\bt;1)$ is convergent for large enough $-\mathrm{Re}\,\bt$ by ratio test on
the series \eqref{eq:Givental_I} without the prefactor.
%\cite[Lemma 4.2]{MR2553377},
By specializing to $\z=1$ 
and using the definition of $\QA$,
one obtains
\begin{align} \label{eq:factorization2}
 \Phi(\bt, \btau;1)
  = \QA \lb I(\bt;1), I(\btau;1) \rb.
\end{align}
By combining \pref{eq:factorization2} with
\pref{eq:Givental_MS},
one obtains
\begin{align} \label{eq:Phi2}
 \Phi(\bt, \btau;1)
  = \QA \lb L^{-1}(\bt;1) F(\bt) \bsone, L^{-1}(\btau;1) F(\btau) \bsone \rb.
\end{align}
Since $L$ is the fundamental solution
for the flat connection $\nablaB$,
the function
$\Phi(\bt, \btau;1)$ is obtained by
parallel-transporting $F(\bt) \bsone \in (\HB)_\bt$
and $F(\btau) \bsone \in (\HB)_\btau$
to the fiber at the same point
and taking the pairing $\QB$ at that point
(the result does not depend on the choice of the point
since $\QB$ is $\nablaB$-parallel).
By sending \pref{eq:Phi2} by \pref{eq:CMS},
one obtains
\begin{align} \label{eq:period}
(2\pi\sqrt{-1})^{n-1} \int_Y I(\bt;-1) I(\btau;1)
  = (-1)^{(n-1)(n-2)/2}\int_\Yv \Omega_\bt \wedge \Omega_\btau.
\end{align}

Assume that $P(\alpha_1, \ldots, \alpha_m)
  = (\alpha_1+\cdots+\alpha_m) Q(\alpha_1, \ldots, \alpha_m)$ for a polynomial $Q$
  and take $R(t_1, ..., t_r) := Q(\sum_{i=1}^r a_{i,1} t_i, ..., \sum_{i=1}^r a_{i,m} t_i) )$
  where $a_{i,j}$ are integers uniquely satisfying $\chi _j = \sum_{i=1}^r a_{i, j} \bp _i$.
By differentiating \pref{eq:period}
by $R(\partial_{t_1}, \ldots, \partial_{t_r})$
and setting $\btau = \bt$,
we obtain
toric residue mirror symmetry
for polynomials of the form
$
P(\alpha_1, \ldots, \alpha_m)
  = (\alpha_1+\cdots+\alpha_m) Q(\alpha_1, \ldots, \alpha_m).
$

\section{Martin's formula}
 \label{sc:Martin}

\subsection{}

We use the same notations $G$,
$T$,
$\W$,
and $\Delta$
for a reductive algebraic group,
a maximal torus,
the Weyl group,
and the set of roots
as in \pref{sc:correlator}.
Let $W$ be an affine scheme
with $G$-action,
and fix a character $\theta$ of $G$.
We write the line bundle on $W \GIT T$
associated with $\alpha \in \Delta$
as $L_\alpha$, and set
\begin{align}
 e \coloneqq \prod_{\alpha \in \Delta} c_1(L_\alpha)
  \in H^{2 |\Delta|}(W \GIT T; \bZ).
\end{align}
We write the natural projection and inclusion as
%\begin{align*}
%\begin{psmatrix}
% W^{G\text{-ss}}/T & W^{T\text{-ss}} / T \\
% W^{G\text{-ss}}/G,
%\end{psmatrix}
%\psset{shortput=nab,nodesep=2mm,hooklength=2mm,hookwidth=-2mm,offset=-1mm}
%\ncline[arrows=H->]{1,1}{1,2}^{\iota}
%\ncline[arrows=->]{1,1}{2,1}^{\pi}
%\end{align*}
and say that
$
 \atilde \in H^*(W \GIT T)
$
is a \emph{lift} of
$
 a \in H^*(W \GIT G)
$
if
$
 \pi^* a = \iota^* \atilde.
$

\begin{theorem}[{Martin \cite{Martin}}] \label{th:Martin}
If $\atilde$ is a lift of $a$,
then one has
\begin{align}
 \int_{W \GIT G} a
  = \frac{1}{|\W|} \int_{W \GIT T} \atilde \cup e.
\end{align}
\end{theorem}

\subsection{}

Let $\Mat(r,n) \cong \bA^{r \times n}$ be the space of $n \times r$ matrices,
which is considered as the space of linear maps
from an $r$-dimensional vector space
to an $n$-dimensional vector space.
It has a natural action of $\GL_r$,
and the GIT quotient
$
 \Gr(r,n) \coloneqq \Mat(r,n) \GIT \GL_r
$
is the Grassmannian of $r$-spaces in an $n$-space.

\subsection{}

When $W = \Mat(r,n)$ and $G = \GL_r$,
one has
\begin{align}
% \Gr(r,n) \coloneqq
 W \GIT G &\cong \Gr(r,n), \\
% \bP \coloneqq
 W \GIT T &\cong (\bP^{n-1})^r
\end{align}
and
\begin{align}
 H^*(\Gr(r,n)) &\cong \bC[\sigma_1, \ldots, \sigma_r] / (h_{n-r+1}, \ldots, h_n), \\
 H^*((\bP^{n-1})^r) &\cong \bC[x_1, \ldots, x_r] / (x_1^n, \ldots, x_r^n),
\end{align}
where
$
 \sigma_i
  = \sigma_i(x_1,\ldots,x_r)
  \in \bC[x_1,\ldots,x_r]^{\fS_r}
$
are elementary symmetric functions and
$
 h_i
  = h_i(x_1,\ldots,x_r)
  \in \bC[x_1,\ldots,x_r]^{\fS_r}
  = \bC[\sigma_1,\ldots,\sigma_r]
$
are complete symmetric functions.
Martin's formula in this case gives
\begin{align}
 \int_{\Gr(r,n)} P(x_1, \ldots, x_r)
  &= \frac{1}{r!} \int_{(\bP^{n-1})^r}
   \prod_{i \ne j} (x_i - x_j)
   P(x_1,\ldots,x_r) \\
  &= \frac{(-1)^{r(r-1)/2}}{r!} \int_{(\bP^{n-1})^r}
   \bsDelta^2 \cup
   P(x_1,\ldots,x_r)
\end{align}
for any
$
 P(x_1,\ldots,x_r) \in \bC[x_1,\ldots,x_r]^{\fS_r}
$
where
$
 \bsDelta \coloneqq \prod_{1 \le i < j \le r} (x_i - x_j).
$

\subsection{}

The equivariant cohomology ring of $\Gr(r,n)$
with respect to the natural action
of the diagonal maximal abelian subgroup
$H \subset  \GL_n$
is presented as
\begin{align}
 H_H^\bullet(\Gr(r,n);\bC)
  \cong \bC[\sigma_1, \ldots, \sigma_r, \lambda_1 , \ldots, \lambda_n]
    \left/ \lb  h_{n-r+1}(\sigma, \lambda), \ldots, h_n(\sigma, \lambda) \rb \right.,
\end{align}
where $h_i$ is the degree $2i$ part of $c^H(\cS) c^H(\cQ) - \prod_{i=1}^n (1 + \lambda_i)$.
Here $\cS$ and $\cQ$ are the tautological bundle and the universal quotient bundle respectively,
and $c^H(-)$ stands for the $H$-equivariant total Chern class.
Note that
$\sigma_i \coloneqq c^H_i(\cS)$
is the elementary symmetric function of the $H$-equivariant Chern roots $x_1, \ldots, x_r$ of $\cS$,
and $c^{H}_i(\cQ)$ for $i = 1, \ldots, n-r$
are expressed in terms of $\sigma_1, \ldots, \sigma_r$ and $\lambda_1, \ldots, \lambda_n$
by the condition $h_1 = \cdots = h_{n-r} = 0$.
Martin's formula gives
\begin{multline} \label{eq:Martin_Gr}
 \int_{\Gr(r,n)}^H P(\sigma_1,\ldots,\sigma_r) \\
 = \sum_{1 \le i_1 < i_2 < \cdots i_r \le n}
  \Res_{\bx = (\lambda_{i_1},\ldots,\lambda_{i_r})}
   P(\sigma_1,\ldots,\sigma_r)
   \prod_{i \ne j} (x_i-x_j)
   \frac{dx_1 \wedge \cdots \wedge dx_r}
   {\prod_{i=1}^r \prod_{j=1}^n (x_i - \lambda_j)}.
\end{multline}

\section{Quasimap spaces for GIT quotients}
 \label{sc:QS}

\subsection{}\label{C1}

Let $G$ be a reductive algebraic group
acting on an affine variety $W$
and fix a character $\theta$ of $G$.
In this paper,
we will always assume the following:
\begin{enumerate}
 \item
Semi-stability implies stability.
 \item
The semi-stable locus $W^\semistable$ is smooth and non-empty.
 \item
The $G$-action on $W^\semistable$ is free
(however,
see \cite{MR3412343}
for allowing finite non-trivial stablizers).
 \item
The codimension of the unstable locus $W \setminus W^\semistable$
is greater than one.
\end{enumerate}
The GIT quotient is defined by
$
 W \GIT G
  \coloneqq W^\semistable/G,
$
which is an open substack of $[W/G]$.

\subsection{}

A map
$u \colon \bP^1 \to [W/G]$
to the quotient stack $[W/G]$
is pair $(P, \utilde)$
of a principal $G$-bundle $P \to \bP^1$
and a $G$-equivariant map
$\utilde \colon P \to W$.
It is called a \emph{quasimap}
if the generic point of $\bP^1$ is mapped
to $W \GIT G \subset [W/G]$.
A point in the inverse image of the unstable locus
will be called a \emph{base point}.

\subsection{}

For a quasimap $u \colon \bP^1 \to [W/G]$
and a $G$-equivariant line bundle $L$ on $W$,
the pull-back $\utilde^* L$
is a $G$-equivariant line bundle on $P$,
which descends to a line bundle $u^* L$ on $\bP^1$.
The \emph{degree}
of a quasimap $u \colon \bP^1 \to [W/G]$
is the map $\bd \colon \Pic^G W \to \bZ$
sending $L \in \Pic^G W$ to $\deg u^*L$.

\subsection{} \label{Quasimap}

An \emph{isomorphism} of quasimaps
$u = (P, \utilde)$ and $u' = (P', \utilde')$
is an isomorphism
$\varphi \colon P \to P'$
of principal $G$-bundles
such that $\utilde = \utilde' \circ \varphi$.
By \cite[Theorem 7.1.6]{MR3126932},
the moduli functor
for quasimaps of degree $\bd$
is representable by a Deligne-Mumford stack,
which will be denoted by $\Q(W \GIT G; \bd)$.
This stack is denoted by
$
 \operatorname{Qmap}_{0,0}(W \GIT G, \bd; \bP^1)
$
in \cite[\S 7.2]{MR3126932} and
$
 \operatorname{QG}^{0+}_{0,0,\bd}(W \GIT G)
$
in \cite[Section 2.6]{MR3272909}.
Note that $\Q(W \GIT G)$ depends
not only on $W \GIT G$ and $\bd$
but also on $W$, $G$, and $\theta$.

\subsection{}

Let $\Qone(W \GIT G; \bd) \subset \Q(W \GIT G; \bd)$
be the substack parametrizing quasimaps
such that
$u|_{\bP^1 \setminus \{ 0 \}}$
is a constant map to $W \GIT G$.
This implies that
$0 \in \bP^1$ is a base point of length $\bd(\theta)$
by \cite[Lemma 7.1.2]{MR3126932}.
%Note that the length of a base point is the degree of the divisor $u^{-1}[(W \setminus W^\semistable) / G] \subset \bP^1$ at the base point. 
This stack is denoted by
$
 \Q_{0,0+\bullet}(W \GIT G, \bd)_0
$
in \cite[Section 4.1]{MR3272909}.
There is a natural map
$
 \ev \colon \Qone(W \GIT G; \bd) \to W \GIT G,
$
called the \emph{evaluation map},
which sends $u \in \Qone(W \GIT G; \bd)$
to $u(\infty) \in W \GIT G$.

\subsection{}

There is a natural $\bGm$-action
on $\Q(W \GIT G;\bd)$
coming from the standard $\bGm$-action on $\bP^1$.
As described in \cite[Section 4.1]{MR3272909},
the fixed locus of this action is identified with the coproduct
\begin{align}
 \coprod_{\bd_1 + \bd_2 = \bd}
  \Qone(W \GIT G; \bd_1)
% {\mathop{\!\! \ _{\ev} \! \times_{\ev}}\limits_{\tiny W \GIT G}}
  \times_{W \GIT G}
  \Qone(W \GIT G; \bd_2)
\end{align}
of fiber products
with respect to the evaluation map.

\subsection{}\label{C2}

If $W$ has at worst lci singularity,
then $\Q(W \GIT G;\bd)$ has a canonical perfect obstruction theory,
which allows one to define the \emph{virtual fundamental cycle}.
The canonical perfect obstruction theory is $\lb \mathbb{R} \pi_* \ev^* T_{[W/G]} \rb^\vee$, where $T_{[W/G]}$ is the tangent complex of $[W/G]$, $\ev: \Q(W \GIT G;\bd) \times \bP^1 \to [W/G]$ is the evaluation map, and $\pi: \Q(W \GIT G;\bd) \times \bP^1 \to \Q(W \GIT G;\bd) $ is the first projection; see Theorem 7.2.2 of \cite{MR3272909}
or Section 5.
The virtual fundamental cycle
is an element of the Chow group of $\Q(W \GIT G;\bd)$
whose degree is given by the virtual dimension
\begin{align}
 \vdim \Q(W \GIT G;\bd) = \pair{\bd}{\det T_W}+ \dim W \GIT G.
\end{align}

\subsection{}

Since the stack $\Qone(W \GIT G; \bd)$
is the union of connected components
of the fixed locus of the $\bGm$-action,
it has a perfect obstruction theory
\begin{comment}
$
 \lb
  \left.
   E_{\Q(W \GIT G; \bd)}^\bullet
  \right|_{\Qone(W \GIT G; \bd)}
 \rb^{\bGm}
$
\end{comment}
inherited from $\Q(W \GIT G; \bd)$.
The virtual push-forward
\begin{align}
 \ev_*^\virt(-) \coloneqq
  \PD \lb \ev_* \lb (-) \cap \ld \Qone(W \GIT G; \bd) \rd^\virt \rb \rb
%   \colon H^\bullet(\Qone(W \GIT G; \bd)) \to H^{\bullet-\vdim}(W \GIT G)
\end{align}
along the evaluation map
$
 \ev \colon \Qone(W \GIT G; \bd) \to W \GIT G
$
allows one to define the \emph{$I$-function}
\begin{align}
 I(\bt ;\z) \coloneqq e^{\inner{\bp}{\bt}/\z} \sum_{\bd \in \Eff(W \GIT G)}
  e^{\inner{\bd}{\bt}} I_\bd
\end{align}
by
\begin{align}
 I_{\bd} \coloneqq \ev^\virt_* \lb
  \frac{1}
  {\Eul^{\bGm} \lb N^{\mathrm{virt}}_{\Qone(W \GIT G; \bd)/\Q(W \GIT G; \bd)} \rb}
    \rb,
\end{align}
where
\begin{comment}
$\bt \in \Pic (W\GIT G)$,
$\inner{\bp}{\bt} = c_1(\bt) \in H^2(W \GIT G)$,
and
\end{comment}
the denominator is the $\bGm$-equivariant Euler class
of the virtual normal bundle.
Here $\bp$ is a basis of $H^2(W \GIT G)$, and
$\bt$ is the coordinate of $H^2(W \GIT G)$ corresponding to $\bp$.

An $H$-action on $W$
commuting with the $G$-action
induces an $H$-action on $\Qone(W \GIT G; \bd)$,
which allows one to define
the $H$-equivariant $I$-function of $W\GIT G$.

\subsection{}
There exits a $G$-space $V$ with a $G$-equivariant closed embedding
$W\hookrightarrow V$.
Let
$
 \fraku \colon \Q (W \GIT G; \bd) \times \bP^1 \to [W/G]
$
be the universal quasimap.
It consists of a principal $G$-bundle
$\cP$ on $\Q (W \GIT G; \bd) \times \bP^1$
and a $G$-equivariant morphism
$\widetilde{\fraku} \colon \cP \to W$.
Let
$
 \cP' \coloneqq \cP|_{\Q (W \GIT G; \bd) \times \{ \mathrm{pt} \}}
$
be the restriction of $\cP$ to a fiber of the second projection
$\Q (W \GIT G; \bd) \times \bP^1 \to \bP^1$.
%As a principal $G$-bundle on $Q(W \GIT G; \bd)$,
We write the Chern--Weil homomorphism
defined by $\cP'$ as
\begin{align}
 \CW \colon \bC[\frakg]^G \to H^* \lb \Q (W \GIT G; \bd) \rb.
\end{align}
Note that $\bC[\frakg]^G$ is isomorphic to $\bC[\frakt]^\W$
by Chevalley restriction theorem.
For $P \in \bC[\frakt]^\W$,
we set
\begin{align}
 \la P \ra_{W\GIT G, \bd}
  &\coloneqq
 \int _{\ld \Q(W\GIT G; \bd)\rd^\virt}
  \CW(P), \\
 \la P \ra_{W \GIT G}
  &\coloneqq
 \sum_{\bd \in \Eff(W \GIT G)} e^{\pair{\bd}{\bt}}
  \la P \ra _{W \GIT G, \bd}.
   \label{eq:quasimap_correlator}
\end{align}

\begin{conjecture} \label{cj:main}
 Suppose that $W\subset V$ is the zero locus of $G$ semi-invariant
polynomials $f_i$, $i=1, ..., r$.
Provided with conditions in \S \ref{C1} and \ref{C2}, for any $P \in \bC[\frakt]^\W$,
the generating function \pref{eq:quasimap_correlator}
of quasimap invariants coincides
with the correlation function
\pref{eq:correlator} of the A-twisted gauged linear sigma model
up to an overall sign;
\begin{align} \label{eq:main_conj}
 \cor{P} = \pm \la P \ra_{W \GIT G}.
\end{align}
Here the potential of GLSM is given as a $G$-invariant function 
$\sum_i f_i p_i$ of $V\times \mathbb{A}^r$ where $p_i$ denotes
i-th coordinate of $\mathbb{A}^r$ with R-charge $2$.
\end{conjecture}

\subsection{}

By taking $\cP'$ to be the fiber
over a fixed point of the natural $\bGm$-action
on the domain curve $\bP^1$,
one can define $\bGm$-equivariant quasimap invariants
$\la P \ra_{W \GIT G}^{\bGm}$.
If $W$ has an action of an algebraic torus $H$
commuting with the action of $G$,
then one can define $H \times \bGm$-equivariant quasimap invariants
$\la P \ra_{W \GIT G}^{H \times \bGm}$.

\section{Quasimap spaces for Grassmannians}
 \label{sc:QSG}

\subsection{}

The quasimap space
$
 \Q(\Gr(r,n); d)
$
classifies pairs $(P, u)$
of a principal $\GL_r$-bundle $P$
and a $\GL_r$-equivariant map $u$.
The choice of a principal $\GL_r$-bundle $P$
is equivalent
to the choice of a vector bundle $S$ of rank $r$,
and the choice of a $\GL_r$-equivariant map $u$
is equivalent to the choice of a map
$\cS \to \cO_{\bP^1}^{\oplus n}$,
which is a sheaf injection
since the generic point must go to the semi-stable locus
(but not necessarily a morphism of vector bundles).
The choice of a sheaf injection
$\cS \to \cO_{\bP^1}^{\oplus n}$
is equivalent to the choice of a surjection
$\cO_{\bP^1}^{\oplus n} \to \cQ$,
where $\cQ$ is a coherent sheaf
whose Hilbert polynomial is $d+(n-r)(t+1)$.
This is the same as the Hilbert polynomial
of a locally free sheaf
of rank $n-r$ and degree $d$,
and one has an isomorphism
\begin{align}
 \Q(\Gr(r,n); d)
  \cong \Quot_{\bP^1,d}(\cO_{\bP^1}^{\oplus n}, n-r).
\end{align}

\subsection{}

It is shown in \cite[Lemma 1.2]{MR2110629}
that the subspace $\Qone(\Gr(r,n);d)$ of $\Q(\Gr(r,n);d)$
is decomposed into connected components as
\begin{align}
 \Qone(\Gr(r,n);d)
  = \coprod_{\abs{\bd}=d} \Qone(\Gr(r,n);\bd),
\end{align}
where
$\bd = (d_1, \ldots, d_r)$ runs over elements of $\bN^r$
satisfying
$\abs{\bd} \coloneqq d_1+\cdots+d_r=d, \ d_1\le d_2 \le ... \le d_r $ 
and each connected component is isomorphic
to the partial flag manifold
\begin{align}
 \Qone(\Gr(r,n);\bd) \cong \Fl(m_1,\ldots,m_k,r,n),
\end{align}
where $1 \le m_1 < m_2 < \cdots < m_k = r$
denote the jumping indices;
\begin{align}
 0 \le d_1 = \cdots = d_{m_1}
  < d_{m_1+1} = \cdots = d_{m+2}
  < \cdots.
\end{align}
Let $x_1, \ldots, x_r$ be the Chern roots
of the dual of the universal subbundle on $\Gr(r, n).$
We also define
$
 \abs{\bx} \coloneqq \sum_{i=1}^r x_i
$
and
$\abs{\bd} \coloneqq \sum_{i=1}^r d_i$
for $\bd = (d_1, ..., d_r)$.
The $I$-function can be computed by localization as
\begin{align}
 I_{\Gr(r,n)}(t;\z)
  &= \sum_{\bd \in \bN^r} (-1)^{(r-1)\abs{\bd}}
   e^{(|\bd|+|\bx|/\z)t} I_\bd(\z)
\end{align}
where
\begin{align} \label{eq:I_Gr}
 I_{\bd}(\z)
  &=
   \frac{\prod_{1\le i < j \le r} (x_i-x_j+(d_i-d_j)\z)}
   {\prod_{1 \le i < j \le r} (x_i-x_j)
    \prod_{i=1}^r \prod_{j=1}^n \prod_{l=1}^{d_i} (x_i + l\z)}.
%     \text{ and } \\
%    \bt & = (t_1, ..., t_r).    
\end{align}
As shown in \cite[page 109]{MR2110629},
the $I$-function and the $J$-function agrees
for $\Gr(r,n)$
just as in the case of projective spaces.

\subsection{}
The Hori--Vafa conjecture
\cite{0002222}
proved in \cite{MR2110629}
shows that
the $I$-functions of $(\bP^{n-1})^r$ and $\Gr(r,n)$ are related by
\begin{align} \label{eq:HV_conj}
 I_{\Gr(r,n)}(t;\z) = e^{-\sigma_1(r-1) \pi \sqrt{-1}/\z}
  \left.
   \frac{\cD I_{(\bP^{n-1})^r}(\bt;\z)}{\bsDelta}
  \right|_{t_i = t + (r-1) \pi \sqrt{-1}}
\end{align}
where
\begin{align}
 \cD \coloneqq \prod_{1 \le i<j \le r}
  \lb \z \frac{\partial}{\partial t_i} - \z \frac{\partial}{\partial t_j} \rb.
\end{align}

\subsection{}

As shown in \cite{MR2110629},
the equivariant $I$-function
with respect to the natural action of $H = (\bGm)^n$
on $\Mat(r,n)$
is given by
\begin{align} \label{eq:J_Gr}
 I_{\Gr(r,n)}^H(t;\z)
  &= e^{t \sigma_1/z} \sum_{\bd \in \bN^r} (-1)^{(r-1)|\bd|} e^{|\bd| t}
   \frac{\prod_{1\le i < j \le r} (x_i-x_j+(d_i-d_j)\z)}
   {\prod_{1 \le i < j \le r} (x_i-x_j)
    \prod_{i=1}^r \prod_{j=1}^n \prod_{l=1}^{d_i} (x_i - \lambda_j + l\z)},
\end{align}
and the factorization gives
\begin{align}
 \sum_{d=0}^\infty e^{d \tau}
  \la e^{(t-\tau) \sigma_1/\z} \ra_{\Gr(r,n),d}^{H \times \bGm}
  = \int_{\Gr(r,n)}^H I^H_{\Gr(r,n)}(t;\z) \cup I^H_{\Gr(r,n)}(\tau;-\z).
\end{align}
Here
$
 \sigma_1 = \sum_{i=1}^r x_i
$
is the $H$-equivariant first Chern class of the vector bundle
\begin{align}
 \cS^\vee
  = \lb \Mat(r, n) \times \bC^r \rb \GIT G
\end{align}
on $\Gr(r,n)$,
where the $G$-action on $\bC^r$ is the defining representation.

\subsection{}

Let $\cV$ be an equivariant vector bundle on $\Gr(r,n)$
associated with a representation $V$ of $\GL_r$.
If $\cV$ is globally generated and
$\det \cV \cong \omega_{\Gr(r,n)}^\vee$,
then the zero $Y \coloneqq s^{-1}(0)$
of a general section $s \in H^0(\cV)$
is a smooth Calabi--Yau manifold
by a generalization of the theorem of Bertini
\cite[Theorem 1.10]{MR1201388}.

\subsection{}

Let $[\Mat(r,n)/\GL_r]$ be the quotient stack
containing $\Gr(r,n)$ as an open substack.
The complete intersection $Y \subset \Gr(r,n)$ is an open substack
of $\cY \coloneqq [Z/\GL_r]$,
where $Z \subset \Mat(r,n)$ is the zero
of the map $\stilde \colon \Mat(r,n) \to V$
underlying $s$.
Indeed, $Y$ has a GIT quotient description $Y = Z \GIT \GL_r$,
which allows us to define $\Q(Y;d)$
and its virtual fundamental cycle
as in \pref{sc:QS}.
Let $\cS_\cY^\vee$ be the vector bundle on $\cY$
associated with the defining representation of $\GL_r$.
Any point $p \in \bP^1$ determines
a map $\ev_p \colon \Q(Y;d) \to \cY$
sending $f \colon \bP^1 \to \cY$ to $f(p) \in \cY$, and
the Chern classes
\begin{align}
 \sigma_i \coloneqq c_i \lb \ev_p^* \cS_\cY ^{\vee} \rb, \quad i = 1, \ldots, r
\end{align}
does not depend on the choice of $p \in \bP^1$.
For
$
 P(\sigma_1, \ldots, \sigma_r)
  \in \bC[\sigma_1,\ldots,\sigma_r],
$
we set
\begin{align}
 \la P(\sigma_1, \ldots, \sigma_r) \ra_{Y,d}
  \coloneqq \int_{[\Q(Y;\bd)]^\virt} P(\sigma_1, \ldots, \sigma_r)
\end{align}
and
\begin{align} \label{eq:Gr_ci}
 \la P(\sigma_1, \ldots, \sigma_r) \ra_{Y}
 \coloneqq \sum_{d=0}^\infty e^{dt}
   \la P(\sigma_1, \ldots, \sigma_r) \ra_{Y,d}.
\end{align}
%
\begin{comment}

For $\bd = (d_1, \ldots, d_r) \in (\bZ^{\ge 0})^r$,
we set $\abs{\bd} \coloneqq d_1+\cdots+d_r$.
%
%\begin{theorem}
The fixed locus of the natural $\bGm$-action on $\Q_{(0,0)}(\Gr(r,n);d)$
is decomposed into connected components as
\begin{align}
 \Q_{(0,0)}(\bG;d)^{\bGm} = \coprod_{\bd_1, \bd_2}
  F_{\bd_1} \times_{\Gr(r,n)} F_{\bd_2}
\end{align} 
where the coproduct runs over the set of
$\bd_1, \bd_2 \in (\bZ^{\ge 0})^r$
such that $\abs{\bd_1}+\abs{\bd_2}=d$, and
$F_\bd$ is a partial flag variety.
We also have
\begin{align}
 \int_{\Q_{(0,0)}(\Gr(r,n);d)} \alpha
  &= \sum_{\abs{\bd_1}+\abs{\bd_2}=d}
   \int_{F_{\bd_1} \times_{\Gr(r,n)} F_{\bd_2}}
   \frac{\iota^* \alpha}
   {\Eul \lb N_{F_{\bd_1} \times_{\Gr(r,n)} F_{\bd_2} / \Q_{(0,0)}(\Gr(r,n);d)} \rb}.
\end{align}

\end{comment}

\subsection{}

The equivariant $I$-function of $Y$ is given by
\begin{align}
 I_Y^H(t;\z) =
 \left.
  \sum_{\bd \in \bN^r}
   e^{\inner{(\bd+\bx/\z)}{\bt}}
   I_{\bd}(\bt;\z)
  \right|_{t_i = t + (r-1) \pi \sqrt{-1}},
\end{align}
where
\begin{align} \label{eq:I_d_Gr}
 I_{\bd}(\bt;\z)
  \coloneqq
   \frac{\prod_{\bsdelta \in \Delta(V)} \prod_{l=1}^{\pair{\bsdelta}{\bd}}
    \lb \pair{\bsdelta}{\bx} + l \z \rb
    \prod_{1\le i < j \le r} (x_i-x_j+(d_i-d_j)\z)}
   {\prod_{1 \le i < j \le r} (x_i-x_j)
    \prod_{i=1}^r \prod_{j=1}^n \prod_{l=1}^{d_i} (x_i - \lambda_j + l\z)},
\end{align}
where $\Delta(V)$ denotes the set of weights of $V$ and
$  \pair{\bsdelta}{\bx}$ denotes the first Chern class associated to
the weight $\bsdelta$ (expressed in terms of the fundamental weights $x_1, ..., x_r$ of 
the maximal diagonal torus of $G$).
Localization with respect to the natural $\bGm$-action
on $\Q(\Gr(r,n);d)$
shows
\begin{align} \label{eq:Gr_factorization}
 \la e^{(t-\tau) \sigma_1 /\z} \ra_Y^{H \times \bGm}
  &= \int_Y^H I^H(t;\z) \cup I^H(\tau;-\z)
\end{align}
just as in \pref{eq:Givental_factorization}.

\section{Residue mirror symmetry for Grassmannians}
 \label{sc:RMSG}

\subsection{}

%Motivated by
%\cite{MR3383085, MR3370259},
We define the \emph{abelianized quasimap space}
for $\Gr(r,n)$ by
\begin{align}
 \AQ(\Gr(r,n);d)
  &\coloneqq \coprod_{\abs{\bd}=d} \AQ(\Gr(r,n);\bd), \\
 \AQ(\Gr(r,n);\bd)
  &\coloneqq \Q(\bP^{n-1};d_1) \times \cdots \times \Q(\bP^{n-1};d_r),    \label{eqn:prod}
% \text{ where } \bd = (d_1,\ldots, d_r) \in \bN^r,
\end{align}
where $\bd$ runs over $\bd = (d_1,\ldots, d_r) \in \bN^r$
such that $\abs{\bd} \coloneqq d_1+\cdots+d_r=d$.
An abelianized quasimap
\begin{align}
 \varphi(z_1,z_2) =
 \bigg(
%  \varphi_i(z_1,z_2) =
   \big( \varphi_{i1}(z_1,z_2), \ldots, \varphi_{in}(z_1,z_2) \big) \in \Q(\bP^{n-1};d_i)
 \bigg)_{i=1}^r
\end{align}
defines a genuine map of degree
$d$
if the matrix $(\varphi_{ij}(z_1,z_2))_{i,j}$ has rank $r$
for any $(z_1,z_2) \ne 0$.
%Recall that the dimension of the moduli space of maps is given by
%\begin{align}
% \cM_{g,k}(X;d)
%  = (1-g) (\dim X-3) + \la c_1(X),d \ra + k,
%\end{align}
%which is given by
%\begin{align}
% (1-0)(r(n-r)-3) + n d + 0
%  = r(n-r)+nd-3
%\end{align}
%when $X=\Gr(r,n)$ and $g=k=0$.
%On the other hand,
%one has
%\begin{align}
% \dim \Q(d_1,\ldots,d_r)
%  &= nd+nr-r \\
%  &= \dim \cM_{0,0}(\Gr(r,n);d)+r(r-1)-3 \\
%  &= \dim \cM_{0,0}(\Gr(r,n);d)+|\Delta|-\dim \PGL_2.
%\end{align}
For
$
 P(\sigma_1,\ldots,\sigma_r)
  \in \bC[\sigma_1,\ldots,\sigma_r],
$
we set
\begin{align}
 \la P(\sigma_1,\ldots,\sigma_r) \ra_{\Gr(r,n),\bd}^\ab
  &\coloneqq \frac{1}{r!} 
   \int_{\AQ(\Gr(r,n);\bd)} 
   \prod_{i \ne j} (x_i - x_j)
   P(\sigma_1(x_1,\ldots,x_r),\ldots,\sigma_r(x_1,\ldots,x_r)),
    \label{eq:Gr_correlator1} \\
 \la P(\sigma_1,\ldots,\sigma_r) \ra_{\Gr(r,n),d}^\ab
  &\coloneqq
  \sum_{|\bd|=d}
   \la P(\sigma_1,\ldots,\sigma_r) \ra_{\Gr(r,n),\bd}^\ab, \\
 \la P(\sigma_1,\ldots,\sigma_r) \ra_{\Gr(r,n)}^\ab
  &\coloneqq
  \sum_{d=0}^\infty (-1)^{(r-1)d} q^d
   \la P(\sigma_1,\ldots,\sigma_r) \ra_{\Gr(r,n),d}^\ab.
\end{align}

\subsection{}

If we set $G \coloneqq \GL_r$ and $V \coloneqq \Mat(r,n)$,
where $G$ acts naturally on $V$
and $\bGm$ acts trivially on $V$,
then we have
$
 \Zvec_\bd(x) = \prod_{i \ne j} (x_i-x_j)
$
and
$
 \Zmat_\bd(x) = \prod_{i=1}^r \lb x_i^{-d_i-1} \rb^n,
$
so that \pref{eq:correlator} gives the same result
as \pref{eq:Gr_correlator1};
\begin{align} \label{eq:Gr_ab}
 \cor{P(\sigma_1,\ldots,\sigma_r)}
  = \la P(\sigma_1, \ldots, \sigma_r) \ra^{\ab}_{\Gr(r,n)}.
\end{align}

\subsection{}

We write the ring homomorphism
$
 \bC[\sigma_1,\ldots,\sigma_r] \to \QH(\Gr(r,n))
$
sending
$
 \sigma_i \in \bC[\sigma_1,\ldots,\sigma_r]
$
to
$
 \sigma_i
  \in H^*(\Gr(r,n);\bC)
  \cong \bC[\sigma_1, \ldots, \sigma_r] / (h_{n-r+1}, \ldots, h_n)
$
as
$
 P(\sigma_1,\ldots,\sigma_r) \mapsto \mathring{P}(\sigma_1,\ldots,\sigma_r)
$
just as in the case of $\bP^{n-1}$.

\begin{theorem} \label{th:Gr}
For any
$
 P(\sigma_1,\ldots,\sigma_r)
  \in \bC[\sigma_1,\ldots,\sigma_r],
$
one has
\begin{align}
 \la P(\sigma_1,\ldots,\sigma_r) \ra_{\Gr(r,n)}^\ab
  &= \int_{\Gr(r,n)} \mathring{P}(\sigma_1,\ldots,\sigma_r).
\end{align}
\end{theorem}

\begin{proof}
It follows from \pref{eq:Pr_int} that
\begin{align*}
 &\la P(\sigma_1,\ldots,\sigma_r) \ra_{\Gr(r,n)}^\ab \\
  &\quad = \frac{1}{r!} \sum_{d_1,\ldots,d_r=0}^\infty
   ((-1)^{r-1}q)^{d_1+\cdots+d_r}
   \Res \prod_{i \ne j} (x_i-x_j)
   P(\sigma_1,\ldots,\sigma_r)
   \frac{d x_1}{x_1^{n(d_1+1)}}
   \wedge \cdots \wedge
   \frac{d x_r}{x_r^{n(d_r+1)}} \\
  &\quad= \frac{1}{r!}
   \Res \prod_{i \ne j} (x_i-x_j)
   P(\sigma_1,\ldots,\sigma_r)
   \frac{d x_1}{x_1^n+(-1)^r q}
   \wedge \cdots \wedge
   \frac{d x_r}{x_r^n+(-1)^r q} \\
  &\quad= \frac{1}{r!n^r} \sum_{x_1^n=(-1)^{r-1}q} \cdots \sum_{x_r^n=(-1)^{r-1}q}
   \prod_{i \ne j} (x_i-x_j)
   P(\sigma_1(x_1,\ldots,x_r),\ldots,
   \sigma_r(x_1,\ldots,x_r))
    \label{eq:Gr_proof_1} \\
  &\quad= \int_{\Gr(r,n)} \mathring{P}(\sigma_1,\ldots,\sigma_r),   
\end{align*}
where the last equality is the Vafa-Intriligator formula
\cite[Theorem 4.6]{MR1621570}.
\end{proof}

\subsection{}

\pref{th:Gr} is related to intersection theory
on the moduli space of vector bundles
on a Riemann surface
through a theorem of Witten
\cite{MR1358625},
%(cf. also \cite{MR1133272}),
which states the existence
of a ring isomorphism
$\QH(\Gr(r,n))/(q-1) \simto R(U(r))_{n-r,n}$
from the quantum cohomology of $\Gr(r,n)$ at $q=1$ and
the Verlinde algebra of $U(r)$
at $\SU(r)$ level $n-r$ and $U(1)$ level $n$.
%A proof can be found, e.g., in \cite[Remark 8.3]{MR2350055}.

\subsection{} \label{sc:Gr1}

We define the $\bGm$-equivariant correlator of
$
 P(\sigma_1,\ldots,\sigma_r)
  \in \bC[\sigma_1,\ldots,\sigma_r]
$
by
\begin{align}
 \la P(\sigma_1,\ldots,\sigma_r) \ra_{\Gr(r,n)}^{\ab,\bGm}
  &\coloneqq \left.
  \sum_{\bd \in \bN^r} e^{\inner{\bd}{\bt}}
   \la P(\sigma_1,\ldots,\sigma_r) \ra_{\Gr(r,n),\bd}^{\ab, \bGm}
  \right|_{t_i = t + (r-1) \pi \sqrt{-1}}
\end{align}
where
\begin{multline}
 \la P(\sigma_1,\ldots,\sigma_r) \ra_{\Gr(r,n),\bd}^{\ab,\bGm}
  \coloneqq
   \int_{\AQ(\Gr(r,n);\bd)}^{\bGm}
   \prod_{1 \le i < j \le n} (x_i - x_j) (x_j - x_i + (d_j - d_i) \z) \\
   P(\sigma_1(x_1,\ldots,x_r),\ldots,\sigma_r(x_1,\ldots,x_r)).
\end{multline}

Since $\Q ((\bP^{n-1})^r;  \bd)^{\bG_m} = \prod_{i=1}^r \Q ( \bP^{n-1};  d_i)^{\bG_m} $ under $\Q ((\bP^{n-1})^r;  \bd) = \prod_{i=1}^r\Q ( \bP^{n-1};  d_i) $,
we have a straightforward generalization of \eqref{eqn:decom Phi}:
\begin{align} \label{eqn:prod Pn}
 \sum_{\bd \in \bN^r} e^{\inner{\bd}{\btau}}
  \la e^{\inner{(\bt-\btau)}{\bx}/z} \ra_{(\bP^{n-1})^r,\bd}^{\bGm}
  = \int_{(\bP^{n-1})^r} I_{(\bP^{n-1})^r}(\bt;\z) \cup I_{(\bP^{n-1})^r}(\btau;-\z),
\end{align}

By acting
$
 \cD_\bt \coloneqq \prod_{1 \le i<j \le r}
  \lb \z \partial_{t_i} - \z \partial_{t_j} \rb
$
and
$
 - \cD_\btau \coloneqq \prod_{1 \le i<j \le r}
  \lb - \z \partial_{\tau_i} + \z \partial_{\tau_j} \rb
$
on both sides of \eqref{eqn:prod Pn}
%%\begin{align} \label{eqn:prod Pn}
%% \sum_{\bd \in \bN^r} e^{\inner{\bd}{\btau}}
%%  \la e^{\inner{(\bt-\btau)}{\bx}/z} \ra_{(\bP^{n-1})^r,\bd}^{\bGm}
%%  = \int_{(\bP^{n-1})^r} I_{(\bP^{n-1})^r}(\bt;\z) \cup I_{(\bP^{n-1})^r}(\btau;-\z),
%%\end{align}
one obtains
\begin{align}
 \sum_{\bd \in \bN^r} e^{\inner{\bd}{\bt}}
 &\la
  \prod_{1 \le i < j \le r} \lb x_i - x_j \rb
  \prod_{1 \le i < j \le r} \lb (x_j + d_j \z) - (x_i + d_i \z) \rb
  \cdot e^{\inner{(\bt-\btau)}{\bx}/z}
 \ra_{(\bP^{n-1})^r}^{\bGm} \\
 &= \la e^{\inner{(\bt-\btau)}{\bx}/z} \ra_{\Gr(r,n)}^{\ab, \mathbb{G}_m} .
\end{align}
%%Here \eqref{eqn:prod Pn} is a straightforward generalization of \eqref{eqn:decom Phi}
%%which follows from 
%%$\Q ((\bP^{n-1})^r;  \bd)^{\bG_m} = \prod_{i=1}^r \Q ( \bP^{n-1};  d_i)^{\bG_m} $ under $\Q ((\bP^{n-1})^r;  \bd) = \prod_{i=1}^r\Q ( \bP^{n-1};  d_i) $.
on the left hand side and
\begin{align}
 \int_{(\bP^{n-1})^r} \cD_{\bt} I_{(\bP^{n-1})^r}(\bt;\z)
  \cup (-\cD_{\btau}) I_{(\bP^{n-1})^r}(\btau;-\z)
\end{align}
on the right hand side.
By setting $t_i = t + (r-1) \pi \sqrt{-1}$,
$\tau_i = \tau + (r-1) \pi \sqrt{-1}$ and
using \pref{eq:HV_conj},
one obtains
\begin{align} \label{eq:Gr_factorization1}
 \la e^{(t-\tau) \sigma_1/\z} \ra_{\Gr(r,n)}^{\ab, \bGm}
  &= \frac{1}{r!} \int_{(\bP^{n-1})^r} \Delta \cup I_{\Gr(r,n)}(t;\z) \cup \Delta \cup I_{\Gr(r,n)}(\tau;-\z) \\
  &= \int_{\Gr(r,n)} I_{\Gr(r,n)}(t;\z) \cup I_{\Gr(r,n)}(\tau;-\z),
   \nonumber
\end{align}
where the last equality is Martin's formula \pref{eq:Martin_Gr}.
On the other hand,
localization with respect to the natural $\bGm$-action
on the domain curve gives the factorization
\begin{align} \label{eq:Gr_factorization2}
  \la e^{(t-\tau) \sigma_1/\z} \ra_{\Gr(r,n)}^{\bGm}
  &= \int_{\Gr(r,n)} I_{\Gr(r,n)}(t;\z) \cup I_{\Gr(r,n)}(\tau;-\z).
\end{align}
Together with \pref{eq:Gr_factorization1},
this gives the equality
\begin{align}
 \la e^{(t-\tau) \sigma_1/\z} \ra_{\Gr(r,n)}^{\ab , \bGm}
  = \la e^{(t-\tau) \sigma_1/\z} \ra_{\Gr(r,n)}^{\bGm}
\end{align}
of the abelianized correlator and the ordinary correlator.

For any $P(\bx) \in \bC[x_1,\ldots,x_r]^{\fS_r}$,
the same argument gives
\begin{align} \label{eq:Gr_ana}
 &\la P(\bx) e^{\inner{(\bt-\btau)}{\bx}/z} \ra_{\Gr(r,n)}^{\ab, \bGm} \\
  &\quad= \int_{\Gr(r,n)} \lb \sum_{\bd \in \bN^r}
   P(\bx+\bd \z) I_{\Gr(r,n),\bd}(\bt;\z) \rb \cup
   \lb \sum_{\bd \in \bN^r} I_{\Gr(r,n),\bd}(\btau;-\z) \rb
   \nonumber \\
  &\quad= \la P(\bx) e^{\inner{(\bt-\btau)}{\bx}/z} \ra_{\Gr(r,n)}^{\bGm}
   \nonumber
\end{align}
where $t_i = t + (r-1) \pi \sqrt{-1}$ and 
$\tau_i = \tau + (r-1) \pi \sqrt{-1}$.
By setting $t=\tau$ in \pref{eq:Gr_ana},
one obtains
\begin{align} \label{eq:Gr_ana2}
 \la P(\bx) \ra_{\Gr(r,n)}^{\ab ,\bGm}
  = \la P(\bx) \ra_{\Gr(r,n)}^{\bGm}.
\end{align}
Together with \pref{eq:Gr_ab},
this proves \pref{cj:main}
for Grassmannians.

%\section{Grassmannian complete intersections defined by homogeneous vector bundles}
% \label{sc:GrCI}

\subsection{}

Let $Y \subset \Gr(r,n)$ be the zero locus of a general section
of a globally-generated vector bundle $\cV$ on $\Gr(r,n)$
associated with a representation $V$ of $\GL_r$.
%
%\subsection{}
%
We define the
\emph{abelianized $\bGm$-equivariant Morrison-Plesser class}
of $Y$ by
\begin{align}
 \Phiabz_\bd(Y;\z) \coloneqq
  \prod_{\bsdelta \in \Delta(V)} \prod_{l=1}^{\pair{\bsdelta}{\bd}}
    \lb \pair{\bsdelta}{\bx} + l \z \rb.
\end{align}
For
$
 P \in \bC[\sigma_1, \ldots, \sigma_r],
%  = \bC[x_1,\ldots,x_r]^{\frakS_r},
$
we set
\begin{align} \label{eq:Gr_ci_ab}
 \la P(\sigma_1, \ldots, \sigma_r) \ra_Y^{\ab,\bGm}
  \coloneqq \sum_{\bd \in \bN^r} q^{\abs{\bd}}
   \la
    P(\sigma_1, \ldots, \sigma_r)
    \Phiabz_\bd(Y;\z) v
   \ra_{\Gr(r,n),\bd}^{\ab,\bGm},
\end{align}
where
$
 v \coloneqq
  \prod_{\bsdelta \in \Delta(V)}
    \pair{\bsdelta}{\bx}
$
is the Euler class of the normal bundle of $Y$
in $\Gr(r,n)$.
By the same reasoning
as in \pref{sc:Gr1}
%\pref{eq:Gr_factorization2}
with the insertion of the abelizanized Morrison-Plesser class,
one obtains
\begin{align}
 \la P(\sigma_1, \ldots, \sigma_r) \ra_Y
  = (-1)^{|\Delta(V) |} \cor{P(\sigma_1, \ldots, \sigma_r)}.
\end{align}
Here,
the identification between $q$ and
the Fayet--Illiopoulos parameter $t'$
is given by 
\begin{align} \label{eq:tt'3}
 q = (-1)^{\sum_{\bsdelta \in \Delta(V)} \pair{\bsdelta}{\bfone}} e^{t'}
\end{align}
where
$
 \bfone \coloneqq (1, \cdots, 1) \in \mathbb{N}^r.
$
%one obtains
%\begin{align} \label{eq:Gr_ab_factorization}
% \la e^{(t-\tau) \sigma_1 / \z} \ra_Y^{\ab,\bGm}
%  &= \int_Y I_Y(t;\z) \cup I_Y(\tau;-\z)
%\end{align}
%\end{theorem}
%\begin{align}
% \la e^{(t-\tau) \sigma_1 / \z} \ra_Y^{\ab,\bGm}
%  &=
% \frac{1}{r!} \sum_{i=1}^r \sum_{s_i =1}^{n}
%  \Res_{x_i=\lambda_{s_i}}
%  \frac{
%   \prod_{1 \le i < j \le n}
%    (x_i - x_j)^2
%   \prod_{\bsdelta \in \Delta(V)}  
%    \pair{\bsdelta}{\bx}
%  }
%  {
%   \prod_{s=1}^n
%   \prod_{i=1}^r
%    (x_i-\lambda_s)
%  }
% \nonumber \\
% &\qquad \qquad
% \lb
%  \sum_{\bd \in \bN^r}
%   e^{t \sigma_1/\z}
%   I_{\bd}(t,x,\z)
% \rb
% \lb
%  \sum_{\bd \in \bN^r}
%   e^{-\tau \sigma_1/\z}
%   I_{\bd}(-\tau,x,\z)
% \rb
% \nonumber \\
% &=
%  \int_Y
%   I_Y(t;\z) \cup I_Y(\tau;-\z),
% \label{eq:Gr_ab_factorization}
%\end{align}
%which implies \pref{cj:main} in this case.
%Now the same reasoning as \eqref{eq:Gr_ana2} gives
%%combining \pref{eq:Gr_factorization}
%%and \pref{eq:Gr_ab_factorization},
%%one obtains the equality
%%
%%\begin{theorem} \label{th:Gr_ab}
%\begin{align} \label{eq:GrCI_ab}
% \la P(\sigma) \ra_Y
%  = \la P(\sigma) \ra_Y^\ab
%\end{align}
%for any $P(\sigma) \in \bC[\sigma_1,\ldots,\sigma_r]$.
%of the generating functions
%\pref{eq:Gr_ci}
%and
%\pref{eq:Gr_ci_ab}.
%\end{theorem}
%It is clear from \pref{eq:GrCI_res1} that
%$\la P \ra^\ab_Y$ coincides with $\cor{P}$,
%and \pref{cj:main} holds in this case.

\subsection{}

As an example,
consider the vector bundle of rank 3
on $\Gr(3,5) = \Mat(3,5) \GIT U(3)$
associated with the representation of $U(3)$
determined by the Young diagram
\begin{align}
 \lambda = \Yvcentermath1 \tiny{\yng(2,2,1)}.
\end{align}
This vector bundle is the tensor product $\wedge^2 \cQ(1)$
of the second exterior power $\wedge^2 \cQ$
of the universal quotient bundle $\cQ$ on $\Gr(2,5) \cong \Gr(3,5)$
and the ample generator $\cO(1)$ of the Picard group.
One can immediately see from the Young diagram that
the restriction of the representation of $U(3)$
associated with $\lambda$
to the diagonal maximal torus $T \cong (\bGm)^3$
is the direct sum
$\rho_{1,2,2} \oplus \rho_{2,1,2} \oplus \rho_{2,2,1}$.
The associated line bundle
on the abelian quotient
$(\bP^4)^3$
is given by
$\cO(1,2,2) \oplus \cO(2,1,2) \oplus \cO(2,2,1)$.

The complete intersection in $\Gr(3,5)$
defined by $\wedge^2 \cQ(1)$
is a Calabi--Yau 3-fold of Picard number 1,
which will be denoted by $Y$ henceforth.
%By the Borel-Weil theory,
%the space $H^0(\wedge^2 \cQ(1))$
%gives a representation of $U(5)$
%associated with $\lambda$.
%A basis of $H^0(\wedge^2 \cQ(1))$
%is given by
%\begin{align}
% \bse_{\tiny{\young(ab,cd,e)}}
%  \coloneqq s_\lambda(\bse_a \otimes \bse_b \otimes \bse_c \otimes \bse_d \otimes \bse_e),
%\end{align}
%where
%$\Yvcentermath1 \tiny{\young(ab,cd,e)}$ runs over the set of Young tableaux
%with support $\lambda$,
%$\{ \bse_1, \ldots, \bse_5 \}$ is the basis of
%the defining representation of $U(5)$, and
%$s_\lambda$ is the Young symmetrizer.
%
%Let $x = (x_{ij})_{i,j}$ be the `homogeneous coordinate' of $\Gr(3,5)$.
%The section $f \in H^0(\wedge^2 \cQ(1))$
%can be written as
%\begin{align}
% f
%  &= f_{23}(x) \bsv_2 \wedge \bsv_3 \otimes \det
%   + f_{31}(x) \bsv_3 \wedge \bsv_1 \otimes \det
%   + f_{12}(x) \bsv_1 \wedge \bsv_2 \otimes \det,
%\end{align}
%where $\bsv_i = (x_{i1}, \ldots, x_{i5}) \in \bC^5$ for $i=1,2,3$,
%$S \coloneqq \vspan_{\bC} \lc \bsv_1, \bsv_2, \bsv_3 \rc$
%is the tautological subspace,
%and $\det = \bsv_1 \wedge \bsv_2 \wedge \bsv_3 \in \wedge^3 S^*$.
%\begin{align}
% \varphi_{ij}(z_1,z_2) = \sum_{m=0}^{d_i} a_{ijm} z^m w^{d_i-m},
%\end{align}
%then
%\begin{itemize}
% \item
%$f_{23}$ is homogeneous of degree $d_1+2d_2+2d_3$ in $(z_1,z_2)$,
% \item
%$f_{31}$ is homogeneous of degree $2d_1+d_2+2d_3$ in $(z_1,z_2)$, and
% \item
%$f_{12}$ is homogeneous of degree $2d_1+2d_2+d_3$ in $(z_1,z_2)$.
%\end{itemize}
The Euler class of the normal bundle of $Y$ is
\begin{align}
 v
  \coloneqq (x_1+2x_2+2x_3)
   (2x_1+x_2+2x_3)
   (2x_1+2x_2+x_3),
\end{align}
the abelianized Morrison-Plesser class is
\begin{multline}
 \Phiab(Y;\bsd)
  \coloneqq (x_1+2x_2+2x_3)^{d_1+2d_2+2d_3} \\
   (2x_1+x_2+2x_3)^{2d_1+d_2+2d_3}
   (2x_1+2x_2+x_3)^{2d_1+2d_2+d_3},
\end{multline}
and the generating function for $\sigma_1^3$ is
\begin{align}
 \la \sigma_1^3 \ra_Y^\ab &=
  - \frac{1}{6}
%  \sum_{d_1,d_2,d_3=0}^\infty
  \sum_{d_1=0}^\infty \sum_{d_2=0}^\infty \sum_{d_3=0}^\infty
  q^{d_1+d_2+d_3} \Res (x_1+x_2+x_3)^3 \\
 &\qquad \qquad (x_1-x_2)^2(x_1-x_3)^2(x_2-x_3)^2
  \Phiab(Y; \bsd) v \frac{d x_1}{x_1^{n(d_1+1)}}
   \wedge \frac{d x_2}{x_2^{n(d_2+1)}}
   \wedge \frac{d x_3}{x_3^{n(d_3+1)}} \nonumber \\
 &= \frac{25(1-q)}{(1+q)(1-123q+q^2)}.
\end{align}
This matches the Yukawa coupling of the mirror
computed by Miura
\cite[\S 5.2]{thesis:Miura}.

\subsection{}

When $\cV$ is a direct sum of line bundles,
the mirror of $Y$ is constructed by toric degenerations
\cite{MR1619529,MR1756568}.
It is an interesting problem to compare
the generating function
\eqref{eq:Gr_ci}
with the Yukawa coupling of this mirror.

\section{Bethe/gauge correspondence}
 \label{sc:B/g}

\subsection{}
%Let $V_1$ be an $r$-dimensional Hermitian vector space
%%with the natural action of $U(r)$,
%and
%$W_1$ be an $n$-dimensional Hermitian vector space.
%with the trivial action of $U(r)$.
Let $V_1$ and $W_1$ be Hermitian vector spaces
of dimensions $r$ and $n$.
The unitary group $U(r)$
acts naturally on $V_1$ and trivially on $W_1$,
inducing an action on
$
 T^* \Hom(V_1, W_1)
  \cong \Hom(V_1, W_1) \oplus \Hom(W_1, V_1).
$
The real and complex moment maps for this action
are given by
\begin{align}
 \mu_\bR &\colon \Hom(W_1, V_1) \oplus \Hom(V_1, W_1)
  \to \End(V_1), \quad
 (i_1, j_1) \mapsto \frac{\sqrt{-1}}{2} \lb i_1 i_1^* - j_1^* j_1 \rb, \\
 \mu_\bC &\colon \Hom(W_1, V_1) \oplus \Hom(V_1, W_1)
  \to \End(V_1), \quad
 (i_1, j_1) \mapsto i_1 j_1.
\end{align}
If
$
 (i_1, j_1) \in \mu_\bR^{-1} \lb \zeta \sqrt{-1} \id_{V_1} \rb
$
for $\zeta < 0$,
then $j_1$ is injective.
If $(i_1, j_1) \in \mu_\bC^{-1}(0)$,
then $i_1$ descends to a map
$
 W_1 / \Image j_1 \to V_1.
$
It follows that the hyperK\"ahler quotient
is isomorphic to $T^* \Gr(r,n)$;
\begin{align}
 \left.
  \lb \zeta_\bR^{-1} \lb \zeta \sqrt{-1} \id_{V_1} \rb \cap \mu_\bC^{-1}(0) \rb
 \right/ U(r)
  \cong T^* \Gr(r,n).
\end{align}
This suggests that
the gauged linear sigma model
with the gauge group $U(r)$
and the reprensetation
$
 V \coloneqq \Hom(W_1, V_1) \oplus \Hom(V_1, W_1) \oplus \End(V_1)
$
describes the quantum cohomology of
$T^* \Gr(r,n)$.
Here $\End(V_1)$ is the Lagrange multiplier
for the complex moment map equation, and
the potential is given by
\begin{align}
 V \ni (i_1, j_1, P) \mapsto \tr(P i_1 j_1).
\end{align}
%
%In general,
%we define the \emph{effective potential} by
%\begin{align}
% \Weff(x;t) &\coloneqq \WFI(x;t) + \Wvec(x) + \Wmat(x), \\
% \WFI(x;t) &\coloneqq t \cdot x, \\
% \Wvec(x) &\coloneqq - \frac{1}{2} \sum_{\alpha \in \Delta_+} \alpha(x), \\
%% \Wmat(x) &\coloneqq - \frac{1}{2 \pi \sqrt{-1}} \sum_{i=1}^N
%%  \rho_i(x) \lb \log \rho_i(x) - 1 \rb.
% \Wmat(x) &\coloneqq - \frac{1}{2 \pi \sqrt{-1}} \sum_{i=1}^N
%  \lb \rho_i(x) - \lambda_i \rb \lb \log \lb \rho_i(x) - \lambda_i \rb - 1 \rb.
%\end{align}
%Each term comes from
%the Fayet-Illiopoulos term,
%the vector multiplet,
%and the matter chiral multiplet respectively.
%
Let
$
 H \coloneqq H_1 \times H_2
$
be the product of
\begin{itemize}
 \item
the diagonal maximal torus $H_1$ of $U(n)$,
acting on $\Hom(W_1, V_1)$ and $\Hom(V_1, W_1)$
through the natural action on $W_1$,
and trivially on $\End(V_1)$, and
 \item
the group $H_2 = U(1)$
acting trivially on $\Hom(W_1, V_1)$,
by scalar multiplication on $\Hom(V_1, W_1)$,
and by inverse scalar multiplication on $\End(V_1)$.
\end{itemize}
One has
\begin{align}
 \Zvec_d(x) &= \prod_{1 \le i \ne j \le r} (x_i - x_j), \\
 \Zmat_d(x) &= \prod_{j=1}^n \prod_{i=1}^r (x_i - \lambda_j)^{-d_i-1} \\
   &\qquad \times \prod_{j=1}^n \prod_{i=1}^r (- x_i + \lambda_j - \mu)^{-(-d_i)-1} \\
   &\qquad \times \prod_{1 \le i \ne j \le r} (x_i-x_j+\mu)^{2-(d_i-d_j)-1},
\end{align}
so that the $H$-equivariant correlator of
$
 P
  \in
%   \bC[\sigma_1,\ldots,\sigma_r] =
  \bC[x_1, \ldots, x_r]^{\fS_r}
$
%of gauged linear sigma model without omega background
is given by
\begin{multline}
  \cor{P}^H
 = \frac{1}{r!} \sum_{d_1=0}^\infty \cdots \sum_{d_r=0}^\infty
   ((-1)^{r-1} e^{t})^{d_1+\cdots+d_r} \\
  \Res \Bigg[
   \frac{\prod_{1 \le i \neq  j \le r  } (x_i-x_j)}
    {\prod_{1 \le i , j \le r} (x_i-x_j+\mu)^{(d_i-d_j-1)}} \\
   \frac{\prod_{j=1}^n \prod_{i=1}^r (-x_i+\lambda_j-\mu)^{d_i-1}}
    {\prod_{j=1}^n \prod_{i=1}^r (x_i-\lambda_j)^{d_i+1}}
   P d x_1 \wedge \cdots \wedge d x_r \Bigg],
   \label{eq:gbcor}
\end{multline}
where
$
 \Res
$
denotes the sum of residues at the points
where $x_i$ is one of $\lambda_j$
for $i=1, \ldots, r$ and $j = 1, \ldots, n$
(there are $n^r$ such points).
This can formally be regarded
as an equivariant integration over the projective space
of dimension $\sum_{i=1}^r (d_i+1) - 1$,
and it is an interesting problem to give a geometric interpretation.

The effective potential \pref{eq:Weff}
of this gauged linear sigma model
is given by
\begin{align}
 \Weff(\bx;t)
  &= \WFI(\bx;t') + \Wvec(\bx) + \Wmat(\bx), \\
 \WFI(\bx;t)
  &= t (x_1+\cdots+x_r), \\
 \Wvec(\bx)
  &= - \pi \sqrt{-1} \sum_{1 \le i < j \le r} (x_j - x_i) \\
  &= - \pi \sqrt{-1} \sum_{i=1}^r (2 i - r -1) x_i,
  \nonumber \\
% \Wmat(x) &\coloneqq - \frac{1}{2 \pi \sqrt{-1}} \sum_{i=1}^N
%  \rho_i(x) \lb \log \rho_i(x) - 1 \rb.
 \Wmat(\bx) &=
   - \sum_{i=1}^r \sum_{j=1}^n
  \lb x_i - \lambda_j \rb \lb \log \lb x_i - \lambda_j \rb - 1 \rb \\
  &\qquad \qquad \qquad - \sum_{i=1}^r \sum_{j=1}^n
  \lb - x_i + \lambda_j - \mu \rb \lb \log \lb - x_i + \lambda_j - \mu \rb - 1 \rb
  \nonumber \\
  &\qquad \qquad \qquad - \sum_{i=1}^r \sum_{j=1}^r
  \lb x_i - x_j + \mu \rb \lb \log \lb x_i - x_j + \mu \rb - 1 \rb,
  \nonumber
\end{align}
where $\lambda_j$ and $\mu$ are
equivariant parameters for the actions
of $H_1$ and $H_2$ respectively.
Note that
\begin{align}
 e^{{\partial \Weff}/{\partial x_i}}
  &= e^t \cdot (-1)^{2i-r-1} \cdot \prod_{j=1}^n (x_i-\lambda_j)^{-1}
   \prod_{j=1}^n (-x_i + \lambda_j - \mu)
   \prod_{j \ne i} \frac{x_j-x_i+\mu}{x_i-x_j+\mu} \\
  &= e^{t+n \pi \sqrt{-1}} \prod_{j=1}^n \frac{x_i - \lambda_j + \mu}{x_i-\lambda_j}
   \prod_{j \ne i} \frac{x_i-x_j-\mu}{x_i-x_j+\mu},
\end{align}
so that the equations
%\begin{align}
% e^{\partial_{x_i} \Weff} = 1, \quad i=1,\ldots,r
%\end{align}
$
 e^{\partial_{x_i} \Weff} = 1,
$
$
 i=1,\ldots,r
$
gives
\begin{align} \label{eq:vacuum}
 \prod_{j=1}^n \frac{x_i -\lambda_j}{x_i - \lambda_j + \mu}
  = e^{t + n \pi \sqrt{-1}} \prod_{j \ne i} \frac{x_i-x_j-\mu}{x_i-x_j+\mu}.
\end{align}
By taking the sum over $d_i$
just as in the proof of \pref{cr:IV1},
one obtains
\begin{multline}
 \cor{P}^H
 = \frac{1}{r!}
  \Res \Bigg[
   \frac{1}{\prod_{i=1}^r \lb \lb 1-e^{\partial_{x_i} \Weff} \rb
    \prod_{j=1}^n \lb x_i-\lambda_j \rb \rb}
  \\
   \frac{\prod_{1 \le i \neq  j \le r  } (x_i-x_j)\prod_{1 \le i , j \le r} (x_i-x_j+\mu)}
    {\prod_{i=1}^r \prod_{j=1}^n (-x_i+\lambda_j-\mu)}
  P d x_1 \wedge \cdots \wedge d x_r \Bigg]
%    \\
%  &\qquad \quad d x_{1} \wedge \cdots \wedge d x_{r} 
 \label{eq:gauge_bethe1}
\end{multline}
where
$
 \Res
$
denotes the sum of residues at the roots of the equations
\pref{eq:vacuum}.

\subsection{}

The Heisenberg model,
also known as the homogeneous $X\!X\!X_{\frac{1}{2}}$ model,
is the $\SU(2)$ spin chain model
with Hamiltonian
\begin{align}
 H
  = \sum_{i=1}^n \bsS_i \cdot \bsS_{i+1},
\end{align}
where
$
 \bsS_i
  = (S_i^x, S_i^y, S_i^z)
  = (\sigma_i^x/2, \sigma_i^y/2, \sigma_i^z/2)
$
are halves of Pauli matrices
acting on the $i$-th factor of the Hilbert space
$
 \cH \coloneqq (\bC^2)^{\otimes n}
$
and
\begin{align}
 \bsS_i \cdot \bsS_{i+1}
  \coloneqq S_i^x S_{i+1}^x + S_i^y S_{i+1}^y + S_i^z S_{i+1}^z.
\end{align}
The total spin
\begin{align}
 S^z \coloneqq \sum_{i=1}^n S_i^z
\end{align}
clearly commutes with the Hamiltonian,
and we restrict to the $S^z$-eigenspace $\cH_r \subset \cH$
with eigenvalue $(-n+r)/2$.
We impose the quasi-periodicity condition
\begin{align}
 \bsS_{n+1}
  = e^{\sqrt{-1} \vartheta S_1^z}
%e^{\frac{\sqrt{-1}}{2} \vartheta S_1^z}
   \bsS_1 
   e^{-\sqrt{-1} \vartheta S_1^z}.
%e^{-\frac{\sqrt{-1}}{2} \vartheta S_1^z}.
\end{align}
Introduce variables $\bx = (x_1,\ldots,x_r)$
related to quasi-momenta $\bp = (p_1,\ldots,p_r)$
by
\begin{align}
 e^{\sqrt{-1} p_i}
  = \frac{x_i+\frac{\sqrt{-1}}{2}}{x_i-\frac{\sqrt{-1}}{2}}.
\end{align}
Then $H$-eigenspaces in $\cH_r$
correspond bijectively to solutions of the \emph{Bethe equation}
\begin{align} \label{eq:Bethe}
 \lb \frac{x_i+\frac{\sqrt{-1}}{2}}{x_i-\frac{\sqrt{-1}}{2}} \rb^n
  = e^{\sqrt{-1} \vartheta} \prod_{j \ne i} \frac
  {x_i-x_j+\sqrt{-1}}{x_i-x_j-\sqrt{-1}}
\end{align}
with eigenavlues
$
 n - 2 r + 2 \sum_{i=1}^r \cos p_i.
$
The integrability comes from factorization of many-body S-matrix
into the product of the 2-body S-matrix
given by
\begin{align}
 S(p_i, p_j) = 1 - 2 e^{\sqrt{-1} p_j} + e^{\sqrt{-1} (p_i+p_j)}.
\end{align}
See e.g. \cite{MR2886419} and references therein
for Bethe ansatz for the quasi-periodic Heisenberg model.
The Bethe equation \eqref{eq:Bethe} coincides
with \eqref{eq:vacuum}
under
$
 \lambda_j = \frac{\sqrt{-1}}{2},
$
$
 j=1,\ldots,n,
$
$
 \mu = - \sqrt{-1},
$
and
$
 \vartheta = - \sqrt{-1}  t + n/2.
$
This observation and its generalizations
is called  \emph{Bethe/gauge correspondence}
\cite{MR2570974}.
The relation between classical/quantum cohomology of Grassmannians
and integrable systems is studied
in \cite{MR2782198,1211.1287,MR3118573,1512.07363}.

\section{Quasimaps and instantons}
 \label{sc:instanton}

\begin{figure}
\centering
\begin{align*} 
\begin{psmatrix}
 \\[-5mm]
 V_1 & V_2 & \ \cdots \ & V_n & V_1 \\
 W_1 & W_2 & \ \cdots \  & W_n & W_1
\end{psmatrix}
\psset{shortput=nab,nodesep=2mm,arrows=->}
\ncline{2,1}{2,2}^{B_1}
\ncline{2,2}{2,3}^{B_1}
\ncline{2,3}{2,4}^{B_1}
\ncline{2,4}{2,5}^{B_1}
\nccircle{->}{2,1}{5mm}_{B_2}
\nccircle{->}{2,2}{5mm}_{B_2}
\nccircle[nodesep=0.5mm]{->}{2,4}{5mm}_{B_2}
\nccircle{->}{2,5}{5mm}_{B_2}
\ncline{3,1}{2,1}^{a}
\ncline{3,2}{2,2}^{a}
\ncline{3,4}{2,4}^{a}
\ncline{3,5}{2,5}^{a}
\ncline{2,1}{3,2}_{b}
\ncline{2,2}{3,3}_{b}
\ncline{2,3}{3,4}_{b}
\ncline{2,4}{3,5}_{b}
\ncline{2,4}{3,5}_{b}
\end{align*}
\caption{The chainsaw quiver}
\label{fg:chainsaw}

\begin{align*}
\begin{psmatrix}
 \\[-5mm]
 V_1 & V_2 & \ \cdots \ & V_{n-1} \\
 W_1 & W_2 & \ \cdots \  & W_{n-1} & W_n
\end{psmatrix}
\psset{shortput=nab,nodesep=2mm,arrows=->}
\ncline{2,1}{2,2}^{B_1}
\ncline{2,2}{2,3}^{B_1}
\ncline{2,3}{2,4}^{B_1}
\nccircle{->}{2,1}{5mm}_{B_2}
\nccircle{->}{2,2}{5mm}_{B_2}
\nccircle[nodesep=0.5mm]{->}{2,4}{5mm}_{B_2}
\ncline{3,1}{2,1}^{a}
\ncline{3,2}{2,2}^{a}
\ncline{3,4}{2,4}^{a}
\ncline{2,1}{3,2}_{b}
\ncline{2,2}{3,3}_{b}
\ncline{2,3}{3,4}_{b}
\ncline{2,4}{3,5}_{b}
\end{align*}
\caption{The handsaw quiver}
\label{fg:handsaw}
\end{figure}

\subsection{}

As explained in \cite[Section 2.3]{MR3161283},
the moduli space of framed instantons
on $\bC \times \ld \bC/(\bZ/n \bZ) \rd$
is isomorphic to the Nakajima quiver variety
associated with the \emph{chainsaw quiver}
shown in \pref{fg:chainsaw}.

\subsection{}

Representations of the chainsaw quiver
%in \pref{fg:chainsaw}
satisfying $\dim V_n = 0$
are in one-to-one correspondence
with representations of the \emph{handsaw quiver}
shown in \pref{fg:handsaw}.
It is shown in \cite[Section 2.3]{MR3161283}
(see also
\cite[Section 3]{MR3024827}
for an exposition)
that the Nakajima quiver variety
associated with the handsaw quiver
is isomorphic to the \emph{parabolic Laumon space}
parametrizing flags
\begin{align}
 0 = E_0 \subset E_1 \subset \cdots \subset E_{n-1}
  \subset E_n = W \otimes_\bC \cO_{\bP^1}
\end{align}
of locally free sheaves on $\bP^1$
such that
$
 \rank E_i = \sum_{j \le i} \dim W_j,
$
$
 \deg E_i = - \dim V_i,
$
and the flag at $\infty \in \bP^1$
is equal to the standard flag
$
 0
  \subset W_1
  \subset W_1 \oplus W_2
  \subset \cdots
  \subset W_1 \oplus W_2 \oplus \cdots \oplus W_{n-1}
  \subset W.
$
This coincides with the space of \emph{based quasimaps}
to partial flag varieties,
i.e., quasimaps with specified value at infinity.

\section{Quasimaps and monopoles}
 \label{sc:monopole}

%In \pref{sc:monopole},
%we recall the well-known relation
%between based rational maps and
%framed monopoles on $\bR^3$.

\subsection{}

Let $G$ be a compact Lie group
with a maximal torus $H$.
A monopole on $\bR^3$ is a pair
$(A, \Phi)$ of a connection $A$
on a principal $G$-bundle $P$
and a section $\Phi$ of $P \times_G \frakg$
satisfying the \emph{Bogomolny equation}
\begin{align}
 F_A = * d_A \Phi.
\end{align}
In order for the curvature
to have a finite $L^2$-norm,
it is natural to demand
that the restriction of $\Phi$
to a sphere with large radius
tends to a map to a fixed adjoint orbit
$\cO \cong G / H \cong G_\bC/P$.
The homotopy class $\bsk \in \pi_2(G_\bC/P)$
of the resulting map is called the \emph{charge}
of the monopole.

\subsection{}

A choice of a gauge
satisfying a certain boundary condition at infinity
is called a \emph{framing} of the monopole.
The framed moduli space is a principal $H$-bundle
over the unframed moduli space.
The framed moduli space
has a natural hyperK\"ahler structure
coming from the dimensional reduction
of the anti-self-dual equation in dimension 4.
%This hyperK\"ahler structure is one advantage
%of the framed moduli space.

\subsection{}

Monopoles on $\bR^3$ are related to
\begin{enumerate}
 \item \label{it:monopole_SC}
spectral curves on $T \bP^1$,
 \item \label{it:monopole_Nahm}
Nahm's equation
\begin{align}
 \frac{d T_i}{ds} = \epsilon_{ijk} [T_j, T_k], \quad i=1,2,3
\end{align}
for
$
 T_i \in C^\infty((0,2),\Mat(k,k;\bC)),
$
and
 \item \label{it:monopole_RM}
based quasimaps 
from $\bP^1$ to $G_\bC/P$
of degree $\bsk$.
\end{enumerate}
\pref{it:monopole_SC}
comes from the twistor correspondence
\cite{MR649818,MR709461}, and
\pref{it:monopole_Nahm}
comes from Nahm transform
\cite{MR766754}.
\pref{it:monopole_RM}
%goes back to \cite{MR763752,MR893593}, and
is proved for $\SU(2)$
in \cite{MR769355},
and the general case can be found
in \cite{MR1625475,MR1625471}
and references therein.

%\subsection{}
%
%When $G=\SU(2)$,
%the framed moduli space of monopoles
%with monopole charge
%\begin{align}
% k \coloneqq \frac{1}{4 \pi} \int_{\bR^3} \norm{F_A}^2
%\end{align}
%is a hyperK\"ahler manifold of dimension $4k$,
%which is studied in detail in \cite{MR934202}.

\section{Quasimaps and vortices}
 \label{sc:vortex}

\subsection{}

Let $X$ be a K\"ahler manifold,
$(E, h)$ be a Hermitian vector bundle on $X$,
and $\tau$ be a positive real number.
The Yang--Mills--Higgs functional
sends a pair $(A, \phi)$
of a unitary connection $d_A$ of $(E, h)$
and a section $\phi$ of $E$ to
\begin{align} \label{eq:YMH}
 \YMH(A, \phi)
  = \norm[L^2]{F_A}^2
    + \norm[L^2]{d_A \phi}^2
    + \frac{1}{4} \norm[L^2]{\phi \otimes \phi^* - \tau}^2.
\end{align}
By \cite[Proposition 2.1]{MR1086749},
one has
\begin{multline} \label{eq:YMH2}
 \YMH(A, \phi)
 = 4 \norm[L^2]{F^{0,2}}^2
  + 2 \norm[L^2]{\partialbar_A \phi}^2
  + \norm[L^2]{\sqrt{-1} \Lambda F
  + \frac{1}{2} \phi \otimes \phi^*
  - \frac{\tau}{2}}^2 \\
  + \tau \int_X \sqrt{-1} \tr F \wedge \omega^{[n-1]}
  + \int_X \tr F \wedge F \wedge \omega^{[n-2]}.
\end{multline}
where
$
 \omega^{[k]} \coloneqq \omega^k /(k!)
$
and $\Lambda$ is the dual Lefschetz operator.

\subsection{}

Assume that $X$ is a projective curve,
so that
\begin{align}
 \deg(E) &= \frac{\sqrt{-1}}{2 \pi} \tr F.
\end{align}
Then \pref{eq:YMH2} immediately implies
the Bogomolny--Prasad--Sommerfield inequality
\begin{align} \label{eq:BPS_bound}
 \YMH(A, \phi) \ge 2 \pi \tau \deg(E),
\end{align}
and the equality holds
if and only if the \emph{vortex equation}
\begin{align}
 F^{0,2} &= 0,
  \label{eq:vortex_eq1} \\
 \partialbar_A \phi &= 0,
  \label{eq:vortex_eq2} \\
 - \sqrt{-1} \Lambda F &= \frac{1}{2} \lb \phi \otimes \phi^* - \tau \id_E \rb
  \label{eq:vortex_eq3}
\end{align}
%\begin{align}
%\left\{
%\begin{aligned}
% F^{0,2} &= 0, \\
% \partialbar_A \phi &= 0, \\
% - \sqrt{-1} \Lambda F &= \frac{1}{2} \lb \phi \otimes \phi^* - \tau \bsI \rb
%\end{aligned}
%\right.
%\end{align}
is satisfied.
\pref{eq:vortex_eq1} and
\pref{eq:vortex_eq2} are
holomorphicities for $E$ and $\phi$, and
\pref{eq:vortex_eq3} is a generalization
of the constant central curvature equation.

\subsection{}

By taking the trace of \pref{eq:vortex_eq3} and
integrating over $X$,
one obtains
\begin{align}
 - 2 \pi \deg(E)
  = \frac{1}{2} \norm[L^2]{\phi}^2
   - \frac{1}{2} \tau \rank(E) \vol(X),
\end{align}
so that the condition
\begin{align}
 \tau \ge \frac{4 \pi \deg(E)}{\rank(E) \vol(X)}
\end{align}
is necessary for \pref{eq:vortex_eq3}
to have a solution.

\subsection{}

The \emph{slope} of a holomorphic vector bundle $E$ is defined by
\begin{align}
 \mu(E) = \frac{\deg(E)}{\rank(E)}.
\end{align}
For a holomorphic section $\phi$ of $E$,
we set
\begin{align*}
 \muhat(E) &\coloneqq \sup \lc \mu(E') \relmid
  \text{$E'$ is a reflexive subsheaf of $E$ of rank less than $E$} \rc, \\
 \mu_M(E) &\coloneqq \max \lc \muhat(E), \mu(E) \rc, \\
 \mu_m(E, \phi) &\coloneqq \inf \bigg\{
  \frac{\rank (E) \mu(E) - \rank(E') \mu(E')}{\rank(E)-\rank(E')} \, \bigg| \\
  & \qquad \text{$E'$ is a reflexive subsheaf of $E$ such that
$\rank E' < \rank E$ and $\phi \in \Gamma(E')$} \bigg\}.
\end{align*}
A pair $(E, \phi)$ of a holomorphic vector bundle $E$
and its holomorphic section $\phi$ is said to be \emph{stable} if
\begin{align}
 \mu_M(E) < \mu_m (E,\phi).
\end{align}

\begin{theorem}[{\cite[Theorem 2.1.6]{MR1085139}}]
Let $(E, \phi)$ be a pair of a holomorphic vector bundle
and its holomorphic section.
If there exists a Hermitian metric on $E$
satisfying the vortex equation,
then one has either of the following:
\begin{enumerate}[(i)]
 \item \label{it:Bradlow_stable}
$(E, \phi)$ is stable and satisfies
\begin{align} \label{eq:tau-bound}
 \mu_M < \frac{\tau \Vol(X)}{4 \pi} < \mu_m(\phi).
\end{align}
 \item
$E$ has a direct sum decomposition
$E = E_\phi \oplus E'$,
$\phi$ is an element of $H^0(E_\phi) \subset H^0(E)$,
$(E_\phi, \phi)$ satisfies
%\pref{it:Bradlow_stable}
(i)
above, and 
$E'$ is the direct sum of stable vector bundles
of slope $\tau \Vol(X)/4 \pi$.
\end{enumerate}
\end{theorem}

\begin{theorem}[{\cite[Theorem 3.1.1]{MR1085139}}]
Let $(E, \phi)$ be a stable pair of a holomorphic vector bundle
and its holomorphic section.
Then for any real number $\tau$ satisfying
\pref{eq:tau-bound},
there exists a Hermitian metric on $E$
satisfying \pref{eq:vortex_eq3}.
\end{theorem}

Bradlow proved these results
not only for projective curves
but also for compact K\"ahler manifolds.

\subsection{}

Vortex equation \pref{eq:vortex_eq3} admits
the following generalization,
which also contains Hitchin's self-duality equation
\cite{MR887284} as a special case.
Let
$
 Q=(Q_0, Q_1, s, t)
$
be a quiver and
$M=(M_a)_{a \in Q_1}$
be a collection of vector bundles on $X$
labeled by $Q_1$.
An 
\emph{$M$-twisted $Q$-sheaf on $X$}
is a pair
$
 R=\lb \lb E_v \rb_{v \in Q_0}, \lb \phi_a \rb_{a \in Q_1} \rb
$
of a collection
$(E_v)_{v \in Q_0}$
of vector bundles
labeled by $Q_0$
and a collection
\begin{align}
 \lb \phi_a \rb_{a \in Q_1}
  \in \prod_{a \in Q_1} \Hom \lb E_{s(a)} \otimes M_a, E_{t(a)} \rb
\end{align}
of morphisms labeled by $Q_1$.

Given a collection $(E_v)_{v \in Q_0}$
of holomorphic vector bundles
on a K\"ahler manifold $X$,
another collection
$(M_a)_{a \in Q_1}$
of holomorphic vector bundles on $X$,
a collection
$\sigma = (\sigma_v)_{v \in Q_0}$
of positive real numbers,
and a collection
$\tau = (\tau_v)_{v \in Q_0}$
of real numbers,
the equation
\begin{align} \label{eq:Q_vortex}
 \sigma_v \sqrt{-1} \Lambda F_v
  + \sum_{t(a)=v} \phi_a \circ \phi_a^*
  - \sum_{s(a)=v} \phi_a^* \circ \phi_a
  = \tau_v \id_{E_v}
%  \qquad v \in Q_0
\end{align}
for Hermitian metrics
on $(E_v)_{v \in Q_0}$
is called the
\emph{$M$-twisted quiver $(\sigma, \tau)$-vortex equation}.

The \emph{$(\sigma, \tau)$-degree} and
the \emph{$(\sigma, \tau)$-slope}
of an $M$-twisted $Q$-sheaf $R$ is defined by
\begin{align}
 \deg_{\sigma, \tau}(R)
  &= \sum_{v \in Q_0} \lb \sigma_v \deg E_v - \tau_v \rank E_v \rb, \\
 \mu_{\sigma, \tau}(R)
  &= \frac{\deg_{\sigma, \tau}(R)}
   {\sum_{v \in Q_0} \sigma_v \rank E_v}.
\end{align}
A $Q$-sheaf is \emph{stable}
if one has
$\mu_{\sigma, \tau}(R') < \mu_{\sigma, \tau}(R)$
for any proper subsheaf $R'$.
A $Q$-sheaf is \emph{polystale}
if it is the direct sum of stable $Q$-sheaf
of the same slope.

\begin{theorem}[{\cite[Theorem 3.1]{MR1989667}}]
A $Q$-sheaf $R$ with $\deg_{\sigma, \tau}(R) = 0$ admits
a Hermitian metric satisfying the quiver vortex equation
\pref{eq:Q_vortex}
if and only if $R$ is $(\sigma, \tau)$-polystable.
This Hermitian metric is unique
up to a multiplication by a positive constant
for each stable summand.
\end{theorem}

Quasimaps to $\Mat(r,n) \GIT \GL_r$ corresponds to the case
when the quiver
$
 Q = \lb 1 \to 2 \rb
$
consists of two vertices
and one arrow between them,
%\begin{align*}
%Q = \lb
%$
%\begin{psmatrix}[colsep=20mm,mnode=circle]
% 1 & 2
%\end{psmatrix}
%%\rb
%,
%\psset{nodesep=3pt,arrows=->}
%\ncline{1,1}{1,2}
%$
%\end{align*}
$M_1$ and $M_2$ are the structure sheaves,
$\rank E_1 = r$, and
$E_2$ is the trivial bundle of rank $n$.

%\section{Symplectic vortex equations}

\subsection{}

Note that the map
$
 V \to \End(V),
$
$
 \phi \mapsto \phi \otimes \phi^*
$
appearing in \pref{eq:vortex_eq3}
is the moment map for the natural action of the unitary group $U(V)$
on $V$.
With this in mind,
a generalization
\begin{align}
 * F_A + \mu(\Phi) = \tau \id_E
\end{align}
of the vortex equation \pref{eq:vortex_eq3}
to the case where one has a Hamiltonian action of a compact group $G$
on a K\"ahler manifold $X$ is given in \cite{MR1801657, MR1777853}.
Here $A$ is a connection on a principal $G$-bundle
on a curve $C$,
$\Phi$ is a holomorphic section of $P \times_G X$, and
$\mu \colon X \to \frakg$ is the moment map.
They are used to define invariants
of a symplectic manifold with a Hamiltonian group action
\cite{MR1777853,MR1953239,MR1959059},
which are closely related
to the Gromov--Witten invariants of the symplectic quotient
\cite{MR2198773, MR3221852,MR3348566,MR3376153,MR3416443}.
\cite{MR2217687} use wall-crossing in vortex invariants
to study quantum cohomology of monotone toric varieties
with minimal Chern number greater than or equal to 2.

%\section{Qusimap spaces and affine vortices}

\subsection{}

Let $X$ be a K\"ahler manifold
with a Hamiltonian action
of a compact connected Lie group $G$.
We assume that $X$ is either compact
or equivariantly convex at infinity
with a proper moment map.
We fix an invariant inner product
to identify $\frakg^\vee$ with $\frakg$,
and write the moment map as
$\mu \colon X \to \frakg$.

%\subsection{}

An \emph{affine vortex} is a pair
$(A, u)$ of a connection $A$
on the principal bundle $P = \bC \times G$
and a holomorphic section
$u \colon C \to P \times_G X$
satisfying the \emph{vortex equation}
\begin{align}
 *F_A + \mu(u) = 0.
\end{align}

%\subsection{}

A \emph{gauged holomorphic map}
to $X$ with respect to the complex Lie group $G_\bC$
acting on $X$
is a map to the quotient stack $[X/G_\bC]$.
In other words,
a gauged holomorphic map
from a scheme $C$ to $X$
is a pair $(P, u)$
of a principal $G_\bC$-bundle $P$ over $C$
and a $G_\bC$-equivariant holomorphic map $u \colon P \to X$.

%\subsection{}

If the $G_\bC$-action on $X^\semistable$ is free,
then by \cite[Theorem 1.1]{1301.7052},
there is a natural bijection between
the set of affine $K$-vortices with target $X$
up to gauge equivalence 
and the set of pairs gauged holomorphic maps
such that $u(\infty) \in X^\semistable$.
This is an open substack of the set of quasimaps
such that $\infty$ is not contained in the base locus.
%precisely the inverse image of $\infty \in \bP^1$
%by the evaluation map $\ev \colon \Qone(\Mat(r,n) \GIT \GL_r; d) \to \bP^1$.

%\section{Two-sphere partition functions}
%
%\cite{MR3152749}

\bibliographystyle{amsalpha}
\bibliography{bibs}

\end{document}